\documentclass[11pt,oneside,reqno]{article}

\usepackage[a4paper,width=16cm,height=24cm]{geometry}

\usepackage{amsmath,amsthm,amssymb}
\usepackage{bbm}
\usepackage{mathrsfs}

\usepackage[authoryear,longnamesfirst]{natbib}
\usepackage{array}

\usepackage{filemod}
\usepackage[breaklinks=true,hidelinks]{hyperref}

\usepackage{etoolbox} 

\makeatletter

\def\@noindentfalse{\global\let\if@noindent\iffalse}
\def\@noindenttrue {\global\let\if@noindent\iftrue}
\def\@aftertheorem{%
  \@noindenttrue
  \everypar{%
    \if@noindent%
      \@noindentfalse\clubpenalty\@M\setbox\z@\lastbox%
    \else%
      \clubpenalty \@clubpenalty\everypar{}%
    \fi}}

\theoremstyle{plain}
\newtheorem{theorem}{Theorem}[section]
\AfterEndEnvironment{theorem}{\@aftertheorem}

\AfterEndEnvironment{proposition}{\@aftertheorem}
\newtheorem{lemma}[theorem]{Lemma}
\AfterEndEnvironment{lemma}{\@aftertheorem}
\newtheorem{corollary}[theorem]{Corollary}
\AfterEndEnvironment{corollary}{\@aftertheorem}

\theoremstyle{definition}
\newtheorem{remark}[theorem]{Remark}
\AfterEndEnvironment{remark}{\@aftertheorem}

\def\be#1{\begin{equation*}#1\end{equation*}}
\def\ben#1{\begin{equation}#1\end{equation}}
\def\bes#1{\begin{equation*}\begin{split}#1\end{split}\end{equation*}}
\def\besn#1{\begin{equation}\begin{split}#1\end{split}\end{equation}}

\def\ba#1{\begin{align*}#1\end{align*}}
\def\ban#1{\begin{align}#1\end{align}}

\def\given{\typeout{Command 'given' should only be used within bracket command}}
\newcounter{@bracketlevel}
\def\@bracketfactory#1#2#3#4#5#6{
\expandafter\def\csname#1\endcsname##1{%
\addtocounter{@bracketlevel}{1}%
\global\expandafter\let\csname @middummy\alph{@bracketlevel}\endcsname\given%
\global\def\given{\mskip#5\csname#4\endcsname\vert\mskip#6}\csname#4l\endcsname#2##1\csname#4r\endcsname#3%
\global\expandafter\let\expandafter\given\csname @middummy\alph{@bracketlevel}\endcsname
\addtocounter{@bracketlevel}{-1}}%
}
\def\bracketfactory#1#2#3{%
\@bracketfactory{#1}{#2}{#3}{relax}{1mu plus 0.25mu minus 0.25mu}{0.6mu plus 0.15mu minus 0.15mu}
\@bracketfactory{b#1}{#2}{#3}{big}{1mu plus 0.25mu minus 0.25mu}{0.6mu plus 0.15mu minus 0.15mu}
\@bracketfactory{bb#1}{#2}{#3}{Big}{2.4mu plus 0.8mu minus 0.8mu}{1.8mu plus 0.6mu minus 0.6mu}
\@bracketfactory{bbb#1}{#2}{#3}{bigg}{3.2mu plus 1mu minus 1mu}{2.4mu plus 0.75mu minus 0.75mu}
\@bracketfactory{bbbb#1}{#2}{#3}{Bigg}{4mu plus 1mu minus 1mu}{3mu plus 0.75mu minus 0.75mu}
}
\bracketfactory{clc}{\lbrace}{\rbrace}
\bracketfactory{clr}{(}{)}
\bracketfactory{cls}{[}{]}
\bracketfactory{abs}{\lvert}{\rvert}
\bracketfactory{norm}{\Vert}{\Vert}
\bracketfactory{floor}{\lfloor}{\rfloor}
\bracketfactory{ceil}{\lceil}{\rceil}
\bracketfactory{angle}{\langle}{\rangle}

\newcounter{ctr}\loop\stepcounter{ctr}\edef\X{\@Alph\c@ctr}%
	\expandafter\edef\csname s\X\endcsname{\noexpand\mathscr{\X}}
	\expandafter\edef\csname c\X\endcsname{\noexpand\mathcal{\X}}
	\expandafter\edef\csname b\X\endcsname{\noexpand\boldsymbol{\X}}
	\expandafter\edef\csname I\X\endcsname{\noexpand\mathbbm{\X}}
	\expandafter\edef\csname r\X\endcsname{\noexpand\mathrm{\X}}
\ifnum\thectr<26\repeat

\let\@IE\IE\let\IE\undefined
\newcommand{\IE}{\mathop{{}\@IE}\mathopen{}}
\let\@IP\IP\let\IP\undefined
\newcommand{\IP}{\mathop{{}\@IP}}

\newcount\minute
\newcount\hour
\newcount\hourMins
\def\now{%
\minute=\time%
\hour=\time \divide \hour by 60%
\hourMins=\hour \multiply\hourMins by 60%
\advance\minute by -\hourMins%
\zeroPadTwo{\the\hour}:\zeroPadTwo{\the\minute}%
}
\def\zeroPadTwo#1{\ifnum #1<10 0\fi#1}

\numberwithin{equation}{section}
\allowdisplaybreaks[4]

\renewcommand\section{\@startsection {section}{1}{\z@}%
{-3.5ex \@plus -1ex \@minus -.2ex}%
{1.3ex \@plus.2ex}%
{\center\small\sc\mathversion{bold}\MakeUppercase}}

\def\subsection#1{\@startsection {subsection}{2}{0pt}%
{-3.5ex \@plus -1ex \@minus -.2ex}%
{1ex \@plus.2ex}%
{\bf\mathversion{bold}}{#1}}

\def\subsubsection#1{\@startsection{subsubsection}{3}{0pt}%
{\medskipamount}%
{-10pt}%
{\normalsize\itshape}{\kern-2.2ex. #1.}}

\def\blfootnote{\xdef\@thefnmark{}\@footnotetext}

\makeatother

\renewcommand{\cite}{\citet}

\def\^#1{\ifmmode {\mathaccent"705E #1} \else {\accent94 #1} \fi}
\def\~#1{\ifmmode {\mathaccent"707E #1} \else {\accent"7E #1} \fi}

\edef\-#1{\noexpand\ifmmode {\noexpand\bar{#1}} \noexpand\else \-#1\noexpand\fi}
\def\>#1{\vec{#1}}
\def\.#1{\dot{#1}}

\def\wt#1{\widetilde{#1}}
\def\atop{\@@atop}

\renewcommand{\leq}{\leqslant}
\renewcommand{\geq}{\geqslant}
\renewcommand{\phi}{\varphi}
\newcommand{\eps}{\varepsilon}

\newcommand{\eq}{\eqref}

\newcommand{\dw}{\mathop{d_{\mathrm{W}}}}

\newcommand{\bigo}{\mathrm{O}}
\newcommand{\lito}{\mathrm{o}}

\newcommand{\toinf}{\to\infty}

\newcommand{\Var}{\mathop{\mathrm{Var}}\nolimits}
\newcommand{\law}{\mathscr{L}}

\newcommand{\SL}{\mathrm{SL}}
\newcommand{\Exp}{\mathrm{Exp}}
\newcommand{\Ber}{\mathrm{Ber}}
\newcommand{\Geo}{\mathrm{Geo}}

\newcommand{\bu}{\mathbf{u}}

\newcommand{\bz}{\mathbf{z}}
\def\sp#1{^{(#1)}}

\newcommand{\Cov}{\mathrm{Cov}}

\newcommand{\done}{d_{\mathrm{W}}}
\def\Ind#1{\rI[#1]}

\def\spm{\sp}

\begin{document}

\title{\sc\bf\large\MakeUppercase{
Exponential and Laplace approximation for occupation statistics of branching random walk
}
}
\author{\sc Erol Pek\"oz, Adrian R\"ollin
and Nathan Ross}
\date{\it Boston University, National University of Singapore
and University~of~Melbourne
}

\maketitle

\begin{abstract}
We study occupancy counts for the critical nearest-neighbor branching random walk on the $d$-dimensional lattice, conditioned on non-extinction. For $d\geq 3$,  \cite{Lalley2011} showed that the properly scaled joint distribution of the number of sites occupied by $j$~generation-$n$ particles, $j=1,2,\ldots$, converges in distribution as $n$ goes to infinity, to a deterministic multiple of a single exponential random variable. The limiting exponential variable can be understood as the classical Yaglom limit of the total population size of generation~$n$. Here we study the second order fluctuations around this limit, first, by providing a rate of convergence in the Wasserstein metric that holds for all $d\geq3$, and second, by showing that for $d\geq 7$, the weak limit of the scaled joint differences between the number of occupancy-$j$ sites and appropriate multiples of the total population size converge in the Wasserstein metric to a multivariate symmetric Laplace distribution. We also provide a rate of convergence for this latter result. 
\end{abstract}

\noindent\textbf{Keywords:} Branching random walk; distributional approximation; 
exponential distribution; multivariate symmetric Laplace distribution

\section{Introduction}
Branching random walk (BRW) is a fundamental mathematical model of a population evolving in time and space, which has been intensely studied for more than 50 years due to its connection to population genetics and superprocesses; see, e.g., \cite[Chapter~9]{Dawson2017} and references. Among this literature, the most relevant to our study is \cite[Theorem~5]{Lalley2011}, which states that the exponential distribution arises asymptotically for certain occupation statistics of a critical BRW conditioned on non-extinction. 
Their result is closely related to the classical theorem of \cite{Yaglom1947}, which says that the distribution of the size of a critical Galton-Watson process, properly scaled and conditioned on non-extinction, converges to the exponential distribution. Yaglom's theorem has a large related literature of embellishments and extensions, e.g., \cite{Lyons1995} and \cite{Geiger2000}, give elegant probabilistic proofs, and \cite{Pekoz2011a} give a rate of convergence using Stein's method.

We now define the nearest neighbor critical BRW on the $d$-dimensional integer lattice. At each time step $n=1,2,  \ldots$, every particle generates an independent number of offspring having distribution $X$ with $\IE[X]=1$, $\Var(X)=\sigma^2<\infty$, and each offspring moves to a site randomly chosen from the $2d+1$ sites having distance less than or equal to $1$ from the site of its parent.
We say that a site has \emph{multiplicity $j$ in the $n$th generation} if there are exactly $j$ particles from the $n$th generation at that site.  Starting the process from a single particle at the origin, let $Z_n$ be the number of particles in $n$th generation, and let $M_n(j)$ be the total number of multiplicity $j$ sites in the $n$th generation. 
\cite[Theorem 5]{Lalley2011},
showed that, when $d\geq 3$, there are constants $\kappa_1, \kappa_2 \ldots$ with $\sum_{j\geq1} j\kappa_j=1$ such that, as $n\to\infty$,
\ben{\label{e1}
\law\left(\frac{Z_n}{n}, \frac{M_n(1)}{n}, \frac{M_n(2)}{n},\ldots\Big|Z_n>0\right)\rightarrow \law\bclr{(1,\kappa_1,\kappa_2,\ldots) Z},
}
 where $Z\sim\Exp(\sigma^2/2)$ is exponential rate $\sigma^2/2$,
and the convergence is with respect to the product topology (which is the same as convergence of finite dimensional distributions).  We study the second order fluctuations in this limit, working with the finite dimensional distributions of~\eq{e1} in the $L_1$-Wasserstein metric. More precisely, let $\norm{\cdot}_1$ be the $L_1$-norm on $\IR^r$, let 
\be{
\cH_r=\{h:\IR^r\rightarrow\IR : \text{$\abs{h(x)-h(y)}\leq  \norm{x-y}_1$ for every $x,y\in\IR^r$} \}
}
be the set of $L_1$-Lipschitz continuous functions with constant $1$, and define 
\ben{\label{eq:wassdef}
\dw\clr{\law(W), \law(V)}:=\sup_{h\in \cH_r}|\IE[h(W)-h(V)]|.
}
Our first main result is as follows. Below and throughout the paper, we use~$c$ to represent constants that do not depend on $n$,
 but possibly $\law(X)$, the dimension $d$, and the length of the vector~$r$, and can differ from line to line. We also disregard the pathological case where $\Var(X)=0$.
\begin{theorem}\label{t1}
With the definitions above, and any $r\geq 1$, $d\geq 3$, and offspring variable~$X$ satisfying $\IE[X]=1, \IE[X^3]<\infty, \Var(X)=\sigma^2$, there are positive constants $\kappa_1,\ldots, \kappa_r$, so that for 
$Z\sim\Exp(\sigma^2/2)$,
\be{
\done\left(\law\left(\frac{Z_n}{n}, \frac{M_n(1)}{n},\ldots , \frac{M_n(r)}{n}\Big|Z_n>0\right)
, \law\bbclr{(1, \kappa_1,\ldots , \kappa_r)Z}\right) \leq c n^{-\frac{d-2}{2(d+1)}},
}
for some constant $c$ independent of $n$.
\end{theorem}

Our next result refines Theorem~\ref{t1} when $d\geq7$ and the offspring distribution has finite higher moments.
To state the result, we define the $r$-dimensional symmetric Laplace distribution with covariance matrix $\Sigma$, denoted $\SL_r(\Sigma)$, as follows: If $E\sim \Exp(1)$ is independent of $\bZ$, which is a centered multivariate normal vector with covariance matrix $\Sigma$, then $\law(\sqrt{E} \bZ)=\SL_r(\Sigma)$. More important for our purposes is that the symmetric Laplace distribution arises as the scaled limit of a geometric random sum of i.i.d.\ centered
random variables; see Section~\ref{sec:lapl} for more details.
\begin{theorem}\label{t3}
Recall the definitions above, and let $r\geq 1$, $d\geq 7$, and offspring variable~$X$ satisfy $\IE[X]=1$ and $\IE\bcls{X^{5+\floor{\frac{18}{d-6}}}}<\infty$. For the positive constants $\kappa_1,\ldots, \kappa_r$ from Theorem~\ref{t1}, there exists a non-negative definite matrix $\wt\Sigma$, and a constant $c$ independent of $n$, such that 
\be{
\done\left(\law\bbbbclr{\frac{M_n(1)-\kappa_1 Z_n}{\sqrt{n}} ,\ldots ,\frac{M_n(r)-\kappa_r Z_n}{\sqrt{n}}\Big|Z_n>0},\SL_r\bclr{\wt\Sigma}\right)\leq c n^{-\frac{2d-9}{6(2d+1)}}.
}
\end{theorem} 

Before discussing the ideas behind the proofs of these two theorems, we  make a few remarks.
The limiting covariance matrix $\wt\Sigma$ is a constant multiple ($\Var(X)/2$) of the limit of the (unconditional) covariance matrix of $\clr{M_n(j)-Z_n\IE[M_n(j) ]}_{j=1}^r$, given by Lemma~\ref{lem:14} below. 
We are only able to show the limit exists, and cannot exclude the possibility that $\wt\Sigma$ is degenerate or even zero. 
Where Theorem~\ref{t3} applies, it implies a rate of convergence of $n^{-1/2}$ in Theorem~\ref{t1}. However, we present the results in this way, as it is conceptually natural, and simplifies the presentation of the proofs.  
An interesting open question from our study is what is the minimal dimension for which the convergence in Theorem~\ref{t3} occurs?
The assumption that $d\geq 7$ stems from Lemma~\ref{lem:14}, giving the behavior of the covariance matrix. It may be possible to sharpen some estimates, e.g.,~\eq{eq:221}, but there are others that may be sharp and still require $d\geq 7$, for example, the upper bound of~\eq{eq:221122}, which must be $\lito(1)$ (as $n\to\infty$) for our arguments to go through.

We now turn
to a discussion of the proofs, where a key result is the following rate of convergence for Yaglom's theorem from \cite{Pekoz2011a}.
\begin{theorem}(\cite[Theorem 3.3]{Pekoz2011a})\label{t2}
  With the definitions above, and any offspring variable~$X$ satisfying $\IE[ X]=1, \IE[ X^3]<\infty, \Var(X)=\sigma^2$, we have
 \be{
 \dw\bclr{\law\left(Z_n/n \, |\,Z_n>0\right),\Exp(\sigma^2/2)}\leq \frac{c \log n}{n}.
 }
\end{theorem}

With this result in hand, the basic intuition behind Theorems~\ref{t1} and~\ref{t3} is that the number of multiplicity $j$ sites in generation $n$ is approximately a sum of a random number of conditionally independent random variables. The number of summands is $Z_m$, for a well-chosen $m<n$, and a given summand represents the contribution from only descendants of a single generation $m$ particle.
If $m$ is large, then Theorem~\ref{t2} implies that $Z_m$ will be roughly exponential with large mean; hence approximately geometrically distributed; and, if $d\geq 3$, the summands approximate the true variable because the random walk is transient, and most low occupancy sites consist of particles descended from exactly one individual in generation~$m$; see Lemma~\ref{lem:indwassbd}. Thus the vector of Theorem~\ref{t1} is approximately a geometric sum with small parameter, which, by R\'enyi's Theorem for geometric sums, is close to its mean times an exponential. For Theorem~\ref{t3}, the idea is similar, but the summands need to be centered for the Laplace limit to arise. A first thought is to subtract the mean of the summands, but in fact we must subtract the mean times a variable with mean one that is highly correlated to the summand to get the correct scaling; see Lemma~\ref{lem:14}. In order to obtain rates of convergence for the Laplace distribution, we prove a general approximation result for random sums, Theorem~\ref{thm:mvrenyi} below.

The approach of \cite{Lalley2011} used to obtain \eq{e1} uses a similar idea, but the conditioning is different to ours, and does not seem amenable to obtaining the error bounds necessary for Theorems \ref{t1} and~\ref{t3}.  Here we use couplings via an explicit construction of $(M_n(j)|Z_n>0)$, which is an elaboration of \cite{Lyons1995}, along with Theorem~\ref{t2}, to evaluate the bounds necessary to obtain Theorems~\ref{t1} and~\ref{t3}. We also use the explicit representation in a novel way to compare two different conditionings appearing in our argument; see Lemma~\ref{l2}. The use of the Wasserstein metric is essential in our argument, even if Kolmogorov bounds are the eventual goal (via standard smoothing arguments). 

The organization of the paper is as follows. In the next section we provide constructions and lemmas used to prove Theorems~\ref{t1} and~\ref{t3}.
In Section~\ref{sec:lapl} we state and prove our general Laplace approximation result, and then apply it to prove Theorem~\ref{t3}. Section~\ref{sec:CLT} gives some auxiliary multivariate normal approximation results that are adapted to our setting, and used in the proofs. 

\section{Constructions, moment bounds and proofs}
To prove Theorems~\ref{t1} and~\ref{t3}, we first need to relate $\law(Z_m|Z_n>0)$ to $\law(Z_m|Z_m>0)$ (in Lemma~\ref{l2} below). We use the size-biased tree construction from \cite{Lyons1995}.

\smallskip\noindent{\bf Size-biased tree construction.}
Assume that the tree is labeled and ordered, so if $w$ and $v$ are vertices in the tree from the same generation and $w$ is to the left of $v$, then the offspring of $w$ is to the left of the offspring of $v$, too.  Start in Generation 0 with one vertex $v_0$ and let it have a number of offspring distributed according to a size-biased version $X^s$ of $X$, so that
\[
\IP(X^s=k)=k\IP(X=k)/\IE[X].
\]
Pick one of the offspring of $v_0$ uniformly at random and call it $v_1$. To each of the siblings of $v_1$, attach an independent Galton-Watson branching process with offspring distribution $X$. For $v_1$ proceed as for $v_0,$ that is, give it a size-biased number of offspring, pick one at uniformly at random, call it $v_2$, attach independent Galton-Watson branching process to the siblings of $v_2$ and so on.
For $1\leq j\leq n$, denote by $L_{n,j}$ and $R_{n,j}$  the number of particles in generation~$n$ of this tree that are descendants of the siblings of $v_j$ to the left and right (excluding $v_j$). This gives rise to the size-biased tree. 

From this construction we define another tree~$T_n$ as follows.
 For $1\leq j\leq n$, let  $R'_{n,j}$ be independent random variables with 
\be{
\law(R'_{n,j})=\law(R_{n,j}|L_{n,j}=0).
}
Start with a single ``marked'' particle in generation~0,  represented as the root vertex of~$T_n$, and give this particle $R'_{n,1}$ offspring.  Then choose the leftmost offspring of the marked particle  as the generation 1 marked particle, and give it $R'_{n,2}$ offspring.  To continue, the generation~$j$ marked particle is the leftmost offspring of the marked particle in Generation~$j-1$, and has $R'_{n,j+1}$ offspring.  In addition, every non-marked particle has descendants according to an independent Galton-Watson  tree with offspring distribution $\law(X)$.  Let $T_n$ be the tree generated in this way to generation~$n$. The key fact from \cite[Theorem C$(i)$]{Lyons1995}, is that the distribution of~$T_n$ is the same as the \emph{entire tree} created from an ordinary Galton-Watson process with offspring distribution~$X$ conditional on non-extinction up to Generation $n$. Moreover, this tree and its marked particles can be closely coupled to the size-biased tree and the $v_j$ ``spine" offspring.

Now, let $A_{k,j} = \{ L_{k,j}=0\}$ and let $X_{m,j}, X_{m,j}'$ be random variables that are independent of each other and the size-biased tree constructed above, such that
	\ba{
	\law(X_{m,j}) &= \law(R_{m,j}+L_{m,j}| A_{m,j}), \\
	\law(X_{m,j}') &=\law(R_{m,j}+L_{m,j}| A_{n,j}),
	}
	and let 
	\besn{                                                            \label{39}
		Y_{m,j} &= R_{m,j} +L_{m,j}+ (X_{m,j} - R_{m,j} - L_{m,j}) \Ind{A_{m,j}^c}, \\
		Y'_{m,j} &= R_{m,j} +L_{m,j}+ (X_{m,j}' - R_{m,j} - L_{m,j}) \Ind{A_{n,j}^c}. 
	}
	Define also $Y_m = 1 + \sum_{j=1}^m Y_{m,j}$ and $Y'_m = 1 + \sum_{j=1}^m Y'_{m,j}$. We have the following result.
\begin{lemma}\label{lem:keyests}
Using the notation and definitions of the size-biased tree construction, for $j=1,\ldots,m$ and $k=m,n$, we have
	\ba{
		(i) &\enskip \law(Y_m) = \law(Z_m | Z_m > 0);\ \ \law(Y_m') = \law(Z_m | Z_n > 0);\\
		(ii) &\enskip \text{there is a constant $K$ such that if $n-m>K$, then} \\
		&\enskip
		\IE\bcls{(L_{m,j}+R_{m,j})\Ind{ A_{n,j}\setminus A_{m,j}}}\leq c \bbbcls{ \frac{(n-m)^2}{(n-j+1)^2}+\frac{\log(m-j+2)}{m-j+1} };\\
		(iii) &\enskip \text{$\IE\cls{(R_{m,j}+L_{m,j})\given A_{k,j}} \leq c $;} \\
		(iv) &\enskip 
		\text{$\IE\bcls{X_{m,j} \Ind{A_{m,j}^c}} \leq \frac{c}{m-j+1}$, and  $\IE\bcls{X_{m,j}' \Ind{A_{n,j}^c}} \leq \frac{c}{n-j+1}$.}
	}
\end{lemma}

Before proving the lemma, we state a key result for controlling the conditioning on non-extinction is the following second order version of ``Kolmogorov's estimate" found in \cite[Display between (5) and (6)]{Vatutin1985}.
\begin{lemma}[Kolmogorov's estimate]\label{lem:kolesterr}
If the offspring variable~$X$ satisfies $\IE[X]=1, \IE[X^3]<\infty, \Var(X)=\sigma^2$,
then as $n\to\infty$,
\be{
\IP(Z_n>0)=\frac{2}{n \sigma^2 }+\bigo\bbbclr{\frac{\log^2(n)}{n^2}}.
}
In particular, $n\IP(Z_n>0)\to 2/\sigma^2$.
\end{lemma}

\begin{proof}[Proof of Lemma~\ref{lem:keyests}]
	Statement $(i)$ follows from the key fact stated above from \citet[Proof of Theorem
	C$(i)$]{Lyons1995}. 
	 For $(ii)$, let $X_j$ denote the number of
	siblings of $v_j$, distributed as $\law(X^s-1)$, $I_j\sim\rU\{1,\ldots,X_j+1\}$ be the position of $v_j$ in its siblings (labeled left to right), and for $i=1,\ldots,I_j-1$, let $L_{m,j}\sp{i}$ be the number of offspring in generation $m$ descended from the $i$th sibling (labeled from the left) of $v_j$ and $A_{k,j}\sp{i}$ be the event $\{L_{k,j}\sp{i}=0\}$; $k=m,n$. Then
	 \ban{
	\IE\bcls{L_{m,j} &\Ind{A_{n,j}\setminus A_{m,j}}} \notag \\
		&\leq\IE\bbcls{  \sum_{i,\ell=1}^{I_j-1}L_{m,j}\sp{i} \Ind{A_{n,j}\sp{\ell}\setminus A_{m,j}\sp{\ell}} } \notag \\
		&=\IE\bbcls{\sum_{i=1}^{I_j-1}L_{m,j}\sp{i} \Ind{A_{n,j}\sp{i}\setminus A_{m,j}\sp{i}}} +\IE\bbcls{ \mathop{\sum_{i,\ell=1}^{I_j-1}}_{\ell\not=i}\IE\bcls{ L_{m,j}\sp{i}| I_j, X_j }\IP\clr{A_{n,j}\sp{\ell}\setminus A_{m,j}\sp{\ell} | I_j, X_j }}  \notag  \\
		&\leq \IE\bbcls{\sum_{i=1}^{I_j-1}L_{m,j}\sp{i} \Ind{A_{n,j}\sp{i}\setminus A_{m,j}\sp{i}}} +\IP\clr{Z_{m-j}>0}\IE\bbcls{ \mathop{\sum_{i,\ell=1}^{I_j-1}}_{\ell\not=i}\IE\bcls{ L_{m,j}\sp{i}| I_j, X_j }} \notag \\
		&\leq \IE\bbcls{\sum_{i=1}^{I_j-1}L_{m,j}\sp{i} \Ind{A_{n,j}\sp{i}\setminus A_{m,j}\sp{i}}} +\IP\clr{Z_{m-j}>0}\IE\bcls{X_j \IE[L_{m,j}|X_j]} \notag \\
				&\leq \IE\bbcls{\sum_{i=1}^{I_j-1}L_{m,j}\sp{i} \Ind{A_{n,j}\sp{i}\setminus A_{m,j}\sp{i}}} +\IP\clr{Z_{m-j}>0}\IE[X_j^2] \notag \\
				&\leq \IE\bbcls{\sum_{i=1}^{I_j-1}L_{m,j}\sp{i} \Ind{A_{n,j}\sp{i}\setminus A_{m,j}\sp{i}}} +\frac{ c}{m-j+1},\notag
	 } 
	 where the first inequality is a union bound, the equality is by independence of lineages, the second inequality is because $\law(L_{m,j}\sp{i})=\law(Z_{m-j})$, and the last is by Lemma~\ref{lem:kolesterr}.
 To bound further, we have that, conditional on $X_j, I_j$, using the independence of lineages,
	 \ba{
	 \IE\bcls{L_{m,j}\sp{i}& \Ind{A_{n,j}\sp{i}\setminus A_{m,j}\sp{i}}} \\
	 &=\IE\bbcls{L_{m,j}\sp{i} \Ind{A_{n,j}\sp{i}}\big\vert L_{m,j}\sp{i}>0} \IP\clr{L_{m,j}\sp{i}>0} \\
	 &=\IE\bbcls{\frac{L_{m,j}\sp{i}}{m-j+1} \bbclr{ \IP\clr{Z_{n-m}=0}^{m-j+1}}^{\frac{L_{m,j}\sp{i}}{m-j+1} }\big\vert L_{m,j}\sp{i}>0} \IP\clr{Z_{m-j}>0}(m-j+1) \\
	  &\leq c \IE\bbcls{\frac{L_{m,j}\sp{i}}{m-j+1} \bbclr{ \IP\clr{Z_{n-m}=0}^{m-j+1}}^{\frac{L_{m,j}\sp{i}}{m-j+1} }\big\vert L_{m,j}\sp{i}>0},
	 }
	 where we have used Lemma~\ref{lem:kolesterr}.
The function $g(x)=xa^x$ is $1$-Lipschitz on $(0,\infty)$ for any~$a<1$, so we can apply Theorem~\ref{t2} and use the fact that $\law(L_{m,j}\sp{i})=\law(Z_{m-j})$ to find that for $W\sim\Exp(2/\sigma^2)$,
	\ba{
	\IE\bbcls{\frac{L_{m,j}\sp{i}}{m-j+1}& \bbclr{ \IP\clr{Z_{n-m}=0}^{m-j+1}}^{\frac{L_{m,j}\sp{i}}{m-j+1} }\big\vert L_{m,j}\sp{i}>0} \\
	&\leq \IE\bbcls{W \bbclr{ \IP\clr{Z_{n-m}=0}^{m-j+1}}^{W }} + c\frac{\log(m-j+1)}{m-j+1} \\
	&=c\bbbcls{\bbclr{2\sigma^{-2}-(m-j+1)\log\bclr{\IP\clr{Z_{n-m}=0}}}^{-2}+\frac{\log(m-j+1)}{m-j+1} }\\
	&\leq c\bbbcls{\bbclr{2\sigma^{-2}+(m-j+1)\IP\clr{Z_{n-m}>0}}^{-2}+\frac{\log(m-j+1)}{m-j+1} } \\
	&\leq c\bbbcls{\bbclr{1+\frac{m-j+1}{n-m}-c' \frac{(m-j+1)(\log(n-m)^2}{(n-m)^2}}^{-2}+\frac{\log(m-j+1)}{m-j+1} },
	} 	
	where we have used Lemma~\ref{lem:kolesterr},
	and we now choose $K$ so that 
	$c' \log(n-m)^2/(n-m)<1/2$ whenever $n-m>K$. Using this bound and combining the last three displays, 
we have
	\ban{
	&\IE\bcls{L_{m,j} \Ind{A_{n,j}\setminus A_{m,j}}} \notag \\
		&\leq c\bbbclc{\IE[X_j] \bbcls{\bbclr{1+\frac{m-j+1}{n-m}-c' \frac{(m-j+1)\log(n-m)^2}{(n-m)^2}}^{-2}+\frac{\log(m-j+1)}{m-j+1} }+\frac{1}{m-j+1}} \notag \\
		&\leq c\bbbclc{ \bbbclr{1+\left(\frac{1}{2}\right)\left(\frac{m-j+1}{n-m}\right)}^{-2}+\frac{\log(m-j+1)}{m-j+1} +\frac{1}{m-j+1}} \notag \\
		&\leq c\bbbclc{ \bbclr{1+\frac{m-j+1}{n-m}}^{-2}+\frac{\log(m-j+2)}{m-j+1} }. \label{eq:lpr1}
	}
Now noting that given $I_j$ and $X_j$, $R_{m,j}$ is independent of $A_{m,j}$ and $A_{n,j}$, we easily find from Lemma~\ref{lem:kolesterr} that 
\be{
\IE\bcls{R_{m,j}\Ind{ A_{n,j}\setminus A_{m,j}}}\leq \IE\bbcls{\IE\bcls{R_{m,j}|X_j, I_j}\IP\bclr{A_{n,j}\setminus A_{m,j}| X_j, I_j}}.
}
A union bound implies
\be{
\IP\bclr{A_{n,j}\setminus A_{m,j}| X_j, I_j}\leq \IP\bclr{A_{m,j}^c| X_j, I_j}\leq X_j \IP(Z_{m-j}>0),
}
and clearly $\IE\cls{R_{m,j}|X_j, I_j}\leq X_j$, so altogether, using Lemma~\ref{lem:kolesterr}, 
\be{
\IE\bcls{R_{m,j}\Ind{ A_{n,j}\setminus A_{m,j}}}\leq \IE[X_j^2] \IP(Z_{m-j}>0) \leq \frac{c}{m-j+1}.
}
Combining this with~\eq{eq:lpr1} shows $(ii)$.

For $(iii)$, we show that  for $k=m,n$,
\ben{\label{eq:negcor}
\IE\bcls{(R_{m,j}+L_{m,j})\Ind{A_{k,j}}} \leq c \IE[X_j ]\IP(A_{k,j}),
}
which easily implies the result. To show~\eq{eq:negcor}, 
we use the following correlation inequality: if $f$ is non-decreasing and $g$ is non-increasing, then $\Cov(f(X),g(X))\leq 0$. Note that
\ba{
 0=\IE\bcls{L_{m,j}\Ind{A_{m,j}}|X_j, I_j}&\leq \IE\bcls{L_{m,j}|X_j, I_j} \IE\bcls{\Ind{A_{m,j}}|X_j, I_j}, \\
 \IE\bcls{R_{m,j}\Ind{A_{m,j}}|X_j, I_j}&\leq \IE\bcls{R_{m,j}|X_j, I_j} \IE\bcls{\Ind{A_{m,j}}|X_j, I_j},
}
where the first line is obvious  and the second is because of conditional independence. Since $\IE\cls{R_{m,j}+L_{m,j}|X_j, I_j}=X_j$, we then have
\be{
 \IE\bcls{(R_{m,j}+L_{m,j})\Ind{A_{m,j}}|X_j, I_j}\leq X_j\IP\bclr{A_{m,j} |X_j, I_j}
 	=X_j\IE\bcls{ \IP\clr{Z_{m-j}=0}^{I_j-1} | X_j }.
	}
But we can couple $I_j$ to $X_j$ in such a way that it is non-decreasing with $X_j$, and thus the correlation inequality implies
\be{
\IE\bbcls{X_j\IE\bcls{ \IP\clr{Z_{m-j}=0}^{I_j-1} | X_j }}\leq \IE[X_j] \IE\bcls{ \IP\clr{Z_{m-j}=0}^{I_j-1}}=\IE[ X_j ]\IP(A_{m,j}).
}
Combining the last two displays implies~\eq{eq:negcor} 
for $k=m$. For $k=n$, using similar ideas,
\be{
\IE\bcls{L_{m,j}\Ind{A_{n,j}}| L_{m,j}}=L_{m,j}\IP(Z_{n-m}=0)^{L_{m,j}},
}
and the second factor decreases with $L_{m,j}$ and so 
\be{
\IE\bcls{L_{m,j}\Ind{A_{n,j}}}\leq \IE[ L_{m,j}] \IP(A_{n,j})\leq \IE[X_j] \IP(A_{n,j}).
}
Finally, 
\be{
\IE\bcls{R_{m,j}\Ind{A_{n,j}}| X_j, I_j} = (X_j+1-I_j) \IP(Z_{n-j}=0)^{I_j-1}\leq X_j \IP(Z_{n-j}=0)^{I_j-1}.
}
But again we can couple $(X_j, I_j)$ such that $I_j$ is non-decreasing in $X_j$, and thus
\bes{
\IE\bcls{R_{m,j}\Ind{A_{n,j}}}\leq \IE[ X_j]  \IE\bcls{P(Z_{n-j}=0)^{I_j-1}}=\IE[ X_j ]\IP(A_{n,j}).
}

Finally, to show $(iv)$, we have
	\ba{
	\IE\bcls{X_{m,j}\Ind{A_{m,j}^c}}&= \IE[X_{m,j}]\IP(A_{m,j}^c)=\IE[R_{m,j}+L_{m,j}|A_{m,j}]\IP(A_{m,j}^c), \\
	\IE\bcls{X_{m,j}'\Ind{A_{n,j}^c}}&= \IE[X_{m,j}' ]\IP(A_{n,j}^c)=\IE[R_{m,j}+L_{m,j}|A_{n,j}]\IP(A_{n,j}^c),
	}
	and then $(iv)$ easily follows from $(iii)$
and
\be{
\IP(A_{k,j}^c)=\IE\bcls{\IP\bclr{A_{k,j}^c| X_j}}\leq \IE[ X_j] \IP(Z_{k-j}>0)\leq \frac{c}{k-j+1}. \qedhere
}
\end{proof}

\begin{lemma}\label{l2}
	If $Z_k$ denotes the size of the $k$th generation of a Galton-Watson tree with offspring distribution $X$ satisfying $\IE[X]=1$, and $\IE[X^3]<\infty$, and $m\leq n$ with $n-m\to\infty$, then
	$$\dw(\law(Z_m|Z_m>0),\law(Z_m|Z_n>0))\leq c\bcls{ (n-m) + \log^2(m)}.$$
\end{lemma}

\begin{proof}[Proof of Lemma \ref{l2}]
The coupling definition of Wasserstein distance and Lemma~\ref{lem:keyests}, parts $(i), (ii), (iv)$, yield an upper bound of 
\bes{
		\IE|Y'_m-Y_m|
		&\leq  \sum_{j=1}^m \IE\bcls{X_{m,j}\Ind{A_{m,j}^c} +X_{m,j}'\Ind{A_{n,j}^c}} \\
		&	\hspace{10mm} +\sum_{j=1}^m \IE\bcls{(R_{m,j} + L_{m,j}) \Ind{A_{m,j}^c\cap A_{n,j}}}\\
		& \leq	c\sum_{j=1}^m  
		\bbbclr{\frac{1}{m-j+1}+\frac{(n-m)^2}{(n-j+1)^2}+\frac{\log(m-j+2)}{m-j+1} }\\
		& \leq c\bcls{ (n-m)+\log^2(m)}.   \qedhere
	}	
\end{proof}

To continue, we need a lemma giving some moment information for variables in the BRW.
\begin{lemma} \label{l1}
Assume the definitions and constructions above, and let 
$Y_{n;m}$ denote the number of particles in generation $n$ of the BRW that occupy a site with another generation $n$ particle that has a different ancestor at generation~$m$. We have the following:
	\begin{enumerate}
		\item The limit $\kappa_j := \lim_{n\rightarrow\infty} \IE[M_n(j) ]$ exists, and
		$$|\IE[M_n(j)] -\kappa_j|\leq c n^{1-d/2};$$
		\item $\Var(M_n(j))\leq 1+n\sigma^2$;
		\item $\IE[Y_{n;m}|Z_n>0]\leq cn(1+m\sigma^2)(n-m)^{-d/2}$.
		\end{enumerate}
\end{lemma}
\begin{proof}[Proof of Lemma \ref{l1}] 	 Part 1 follows from the proof of  \cite[Proposition 21]{Lalley2011}, where they obtain
	$$|\IE[M_{n+1}(j)]-\IE[M_{n}(j)] |\leq c n^{-d/2},$$ which means $\IE[M_{n+1}(j)], n=1,2,\ldots$ has a limit and the bound follows since 
	\be{
	\sum_{i=n}^\infty i^{-d/2} \leq c n^{1-d/2}.
	}
Part 2 follows from $M_n(j)\leq Z_n, \IE[Z_n]=1$ and $ \Var(Z_n)=n\sigma^2$ to get $\Var(M_n(j))\leq \IE[Z_n^2]\leq 1+n\sigma^2$. To prove Part 3, we have
		\be{
		\IE[Y_{n;m}|Z_n>0] = \frac{\IE\bcls{Y_{n;m}\Ind{Z_n>0}}}{\IP(Z_n>0)}\leq  \frac{\IE[Y_{n;m}] }{\IP(Z_n>0)}\leq c n (n-m)^{-d/2} \IE[Z_m^2],
		}
		where in the last inequality we have used Lemma~\ref{lem:kolesterr} and \cite[Corollary 20]{Lalley2011}. The lemma follows after noting that $\Var(Z_m)=m \sigma^2$ and $\IE[Z_m]=1$.
\end{proof}
We give a construction of the critical BRW, building from the size-biased tree construction.

\smallskip\noindent{\bf BRW construction.}
To construct $M_n(j)$ conditional on $Z_n>0$, 
first generate $T_n$ from the size-bias tree section above. Since this tree is distributed as the Galton-Watson tree given $Z_n>0$, we  construct the conditional BRW by attaching a random direction to each offspring, chosen uniformly and independently from the $2d+1$ available directions for the nearest-neighbor random walk.
It is obvious that this ``modified'' BRW process has the same distribution as the original conditioned on non-extinction to generation~$n$.  
For the modified process, let $\hat Z_k$ denote the size of generation $k$ and $\hat M_n(j)$ be the number of multiplicity $j$ sites in generation $n$, then, in particular, we have
\be{
\law\bclr{Z_m, Z_n, (M_n(j))_{j=1}^r|Z_n>0}=\law\bclr{\hat Z_m, \hat Z_n,(\hat M_n(j))_{j=1}^r}.
}

\smallskip\noindent{\bf Modified BRW construction.}
A key to our approach is the following lemma that shows the cost of replacing $\hat M_n(i)$ by a sum of a random sum of conditionally independent variables. Given $\hat Z_m$, for $i=2,\ldots,\hat Z_m$, let $Z_{n,m}^i$ be the number of generation~$n$ offspring of the $i$th particle in generation $m$ of the modified BRW construction; here the labelling is left to right (so particle $1$ is always the marked particle), and note these are distributed as the sizes of the $(n-m)$th generations of  i.i.d.\ Galton-Watson trees with offspring distribution~$\law(X)$. Let also $M_{n,m}^i(j)$ be the number of sites having exactly~$j$ generation-$n$ descendants from the generation~$m$ particle labeled~$i$ in the critical BRW construction above, where the counts ignore particles descended from other generation $m$ particles at those sites. Also let $(Z^1_{n,m},M_{n,m}^1(j)$ be an independent copy of $(Z_{n-m},M_{n-m}(j))$. Note that given $\hat Z_m$, $M_{n,m}^1(j), \ldots, M_{n,m}^{\hat Z_m}(j)$ are i.i.d.

\begin{lemma}\label{lem:indwassbd}
For the variables described above and $m<n$,
\be{
\IE\bbbabs{\hat M_n(j)-\sum_{i=1}^{\hat Z_m} M_{n,m}^i(j) }\leq 
c\bclr{ n m (n-m)^{-d/2}+ (n-m)}.
}
\end{lemma}
\begin{proof}
The differences between the two variables are (i) multiplicity $j$ sites with more than $1$ ancestor from generation $m$, (ii) multiplicity $k>j$ sites with exactly $j$ particles descended from some single generation $m$ particle, (iii) the number of multiplicity $j$ sites with only descendants of the first particle of generation $m$, and (iv) $M_{n,m}^1(j)$.
But (i) and (ii) together are bounded by $Y_{n;m}$ (from Lemma~\ref{l1}) and (iii) is bounded by $\sum_{i=m+1}^n  R'_{n,i}$. Thus, using  using Items~1 and~3 of Lemma~\ref{l1} and~\emph{(iii)} of Lemma~\ref{lem:keyests}, we have
\ba{
\IE\bbbabs{\hat M_n(j)-\sum_{i=1}^{\hat Z_m}M_{n,m}^i(j) }& \leq \IE[Y_{n;m}|Z_n>0]+\sum_{i=m+1}^n  \IE[R'_{n,i}]  + \IE[M_{n-m}(j)] \\
	&\leq c \bclr{n(1+m\sigma^2)(n-m)^{-d/2}+ (n-m)}. \qedhere
}
\end{proof}

Before proving Theorem~\ref{t1}, we state and prove a simple lemma.

\begin{lemma}
For any nonnegative random variable $Y$ on the same space as $(Z_j)_{0\leq j\leq n}$ and $m<n$, we have
\besn{\label{eq:diffcond}
\IE[&Y|Z_n>0] \leq \frac{cn}{m}\IE[Y|Z_m>0].
}
\end{lemma}
\begin{proof}
Using Kolmogorov's approximation,
\bes{
\IE[&Y|Z_n>0] \\ 
	&=\frac{\IE[Y1_{Z_n>0}]}{\IP(Z_n>0)}\leq\frac{\IE[Y1_{Z_m>0}]}{\IP(Z_n>0)}=\IE[Y|Z_m>0] \frac{\IP(Z_m>0)}{\IP(Z_n>0)}\leq \frac{cn}{m}\IE[Y|Z_m>0].
}
\end{proof}
\begin{proof}[Proof of Theorem \ref{t1}]
	Fix $n\geq 4$ and $$1\leq m:=\bfloor{ n-n^{3/(d+1)}}=\bfloor{ n\bclr{1-n^{-(d-2)/(d+1) }}}<n.$$
The triangle inequality implies
\ban{
\done&\left(\law\left(\frac{Z_n}{n}, \frac{M_n(1)}{n},\ldots , \frac{M_n(r)}{n}\Big|Z_n>0\right), \law\bclr{(1,\kappa_1,\ldots,\kappa_r) Z}\right) \notag\\
&\leq
\done\left(\law\left(\frac{Z_n}{n}, \frac{M_n(1)}{n},\ldots , \frac{M_n(r)}{n}\Big|Z_n>0\right),\law\left((1,\kappa_1,\ldots,\kappa_r) \frac{Z_m}{m}\Big|Z_n>0\right)\right) \label{eq:11} \\
& \ \ \  +
\done\left(\law\left((1,\kappa_1,\ldots,\kappa_r) \frac{Z_m}{m}\Big|Z_n>0\right), \law\bclr{(1,\kappa_1,\ldots,\kappa_r) Z}\right). \label{eq:11a} 
}
Using Lemma \ref{l2} and Theorem~\ref{t2}, we find~\eq{eq:11a} is upper bounded by $r c (n-m)/m=\bigo(n^{\frac{2-d}{d+1}})$.

From the coupling definition of the Wasserstein metric,~\eq{eq:11} is upper bounded by
\be{
\IE\left|\frac{\hat Z_n}{n}-\frac{\hat Z_m}{m}\right|+\sum_{j=1}^r \IE\left|\frac{\hat M_n(j)}{n}-\frac{\kappa_j\hat Z_m}{m}\right|\leq b_{1}+b_{2}+\sum_{j=1}^r (e_{1,j}+e_{2,j}),
}
where the hat-couplings are those in the BRW description above and
\ba{
	b_1&=\frac{1}{n}\IE\left|\hat Z_n-\hat Z_m\right|, \\
	b_2&=\frac{n-m}{nm} \IE[\hat Z_m], \\
	e_{1,j}&= \frac{1}{n} \IE\left|\hat M_n(j)-\sum_{i=1}^{\hat Z_m}M_{n,m}^i(j)\right|, \\
e_{2,j} 	&= \IE\left|\frac{1}{n}\sum_{i=1}^{\hat Z_m}M_{n,m}^i(j) -\frac{\kappa_j\hat Z_m}{m}\right|. 
}
For $b_1$, we have that
\be{
\left|\hat Z_n-\hat Z_m\right|\leq \sum_{i=m+1}^n R'_{n,i} + \left| \sum_{i=2}^{\hat Z_m} \bclr{Z_{n,m}^i-1}\right|,
}
where recall the $Z_{n,m}^i$ are the number of generation~$n$ offspring of the $i$th particle in generation $m$ of the modified BRW construction, which are distributed as the sizes of the $(n-m)$th generations of  i.i.d.\ Galton-Watson trees with offspring distribution~$\law(X)$.
 Thus, using $(iii)$ of Lemma~\ref{lem:keyests}, conditioning on $\hat Z_m$, and using the fact that $\IE[Z_{n,m}^i]=1$, $\Var(Z_{n,m}^i)=(n-m)\sigma^2$, Cauchy-Schwarz, and then Jensen's inequality, we have
\be{
b_1\leq \frac{c}{n} \bbclr{n-m+ \IE\bbcls{\sqrt{(n-m)\sigma^2 \hat Z_{m}}}}\leq 
	 \frac{c}{n} \bbclr{n-m+\sqrt{(n-m)\sigma^2} \sqrt{\IE[\hat Z_{m}]}}\leq c n^{\frac{2-d}{2(d+1)}}.
}
Now, using~\eq{eq:diffcond} which says $\IE[\hat Z_m]=\IE[Z_m| Z_n>0]\leq c n$, we find $b_2=\bigo(n^{\frac{2-d}{d+1}})$. And Lemma~\ref{lem:indwassbd} implies that~$e_{1,j}=\bigo(n^{\frac{2-d}{2(d+1)}})$.

For $e_{2,j}$, recall that conditional on $\hat Z_m$,
$M_{n,m}^1(j), M_{n,m}^2(j), \ldots$ are i.i.d.\ random variables all having the same distribution as $M_{n-m}(j)$. 
Thus, letting $Z_m$ be distributed as the size of the $m$th generation in a Galton-Watson tree with offspring distribution $X$ that is independent of the $M_{n,m}^i(j)$, we find
\bes{
\frac{m}{cn} e_{2,j} 
	&\leq 
	 \IE\bbbclc{\IE\bbcls{\babs{n^{-1}\sum_{i=1}^{Z_m}M_{n,m}^i(j)-\kappa_jZ_m/m} \Big| Z_m, Z_m>0}\Big|Z_m>0}\\
	&\leq \IE\bbbclc{\IE\bbcls{\bclr{n^{-1}\sum_{i=1}^{Z_m}M_{n,m}^i(j)-\kappa_jZ_m/m}^2\big|Z_m, Z_m>0}^{1/2}\Big|Z_m>0} \\
	&= \IE\bbbclc{\bbcls{n^{-2} Z_m\Var(M_{n-m}(j))+\bclr{n^{-1}Z_m\IE[M_{n-m}(j)] -\kappa_jZ_m/m}^2}^{1/2}\Big|Z_m>0} \\
	&\leq \IE\bbcls{n^{-1} \sqrt{Z_m}\Var(M_{n-m}(j))^{1/2}
			+n^{-1}Z_m\babs{\IE[M_{n-m}(j)] -n\kappa_j/m}\big|Z_m>0}\\
	&\leq n^{-1}\IE\bcls{\sqrt{Z_m}|Z_n>0}\Var(M_{n-m}(j))^{1/2}\\ & \ \ \  +n^{-1}\IE[Z_m|Z_n>0]\bclr{|\IE[ M_{n-m}(j)] -\kappa_j|+  \kappa_j(n-m)/m} \\
	 & \leq c n^{-\frac{d-2}{2(d+1)}},}
where the last line follows by $\IE[\sqrt{Z_m}|Z_n>0]\leq \sqrt{\IE[Z_m|Z_n>0]} \leq cn^{1/2}$--using~\eq{eq:diffcond}--and Parts~1 and~2 of Lemma \ref{l1}, which imply
\[
\Var(M_{n-m}(j))^{1/2}\leq c\sqrt{n-m}\leq c n^{1.5/(d+1)}
\mbox{\, and  \,} |\IE[M_{n-m}(j)]-\kappa_j|\leq c (n-m)^{1-d/2}. \qedhere
\]
\end{proof}

\subsection{Laplace distribution approximation}\label{sec:lapl}

The centered multivariate symmetric Laplace distribution is a cousin to the Gaussian distribution that arises in a number of contexts and applications; see \cite{Kotz2001} for a book length treatment of this distribution. The $r$-dimensional distribution is denoted $\SL_r(\Sigma)$, where the parameter $\Sigma$ is an $r\times r$ positive definite matrix. 
In general, its law is the same as that of $\sqrt{E} \bZ$, where $E\sim \Exp(1)$ and $\bZ$ is a centered multivariate normal vector with covariance matrix $\Sigma$. The covariance matrix of $\SL_r(\Sigma)$ is $\Sigma$, which can thus be thought of as a scaling parameter. The characteristic function is evidently 
\be{
\bu\mapsto \frac{1}{1+ \frac{1}{2}\bu^\top \Sigma \bu},
} 
and from this it's easy to see a number of 
equivalent characterizations in the case $r=1$: if $E_1, E_2$ are independent and $\Exp(1)$, then $E_1-E_2\sim \SL_1(2)$; if $B\sim\Ber(1/2)$ independent of $E_1,E_2$, then $B E_1-(1-B) E_2\sim\SL_1(2)$. The $1$-dimensional density of $\SL_1(2)$ is $\frac{1}{2}e^{- \abs{x}}$ and the multivariate density is given in \cite[(5.2.2)]{Kotz2001} in terms of modified Bessel functions of the 3rd kind.

The symmetric Laplace distribution arises as the limit of a geometric sum. More precisely, we have the  following theorem, which is elementary, using, for example, characteristic functions.
\begin{theorem}\label{thm:renyi}
Let $N_p\sim \Geo(p)$ be independent of $\bX_1,\bX_2,\ldots$, which are i.i.d.\ $r$-dimensional random vectors having mean zero and covariance matrix $\Sigma$. 
Then as $p\to0$,
\be{
p^{1/2} \sum_{i=1}^{N_p} \bX_i \stackrel{d}{\longrightarrow} \SL_r(\Sigma).
}
\end{theorem}

Here we provide a rate of convergence to a generalization of Theorem~\ref{thm:renyi} in a metric amenable to our setting; see also \cite{Pike2014} for a related result when $r=1$. 

\begin{theorem}\label{thm:mvrenyi}
Let $M\geq1$ be a random variable with mean $\mu>1$, independent of $\bX_1,\bX_2,\ldots$, which are i.i.d.\ $r$-dimensional random vectors with zero mean, covariance matrix $\Sigma=(\Sigma_{ij})$, and finite third moments. Then there is a constant $C_r$ depending only on~$r$ such that
\bes{
\done\bbclr{\law\bbclr{\mu^{-1/2}& \sum_{i=1}^{M} \bX_i},  \SL_r(\Sigma)} \\
	&\leq  \mu^{-1/2}\bbbclr{C_r \mu^{1/3}  \IE\bcls{ \norm{\bX_1}_1^3}^{1/3} +\bbclr{\sum_{i=1}^r \sqrt{\Sigma_{ii}}}  \done\bclr{\law(M),\Geo(\mu^{-1})}+ 3.5}.
}
\end{theorem}
\begin{proof}
Let $E\sim \Exp(1), N\sim \Geo(\mu^{-1})$ and $\bZ=\bZ_\Sigma$ be a centered multivariate normal vector with covariance matrix $\Sigma$, with the three variables independent and independent of $M$ and the $\bX_i$. Then, since $\law(\sqrt{E} \bZ)=\SL_r(\Sigma)$,  the triangle inequality implies
\ban{
\done\bbclr{\law\bbclr{\mu^{-1/2} \sum_{i=1}^{M} \bX_i},  \SL_r(\Sigma)}
	&\leq \done\bbclr{\law\bbclr{\mu^{-1/2} \sum_{i=1}^{M} \bX_i}, \law\bclr{ \sqrt{\mu^{-1} M} \bZ}}  \label{eq:cltwass1} \\
	&\qquad +\done\bbclr{ \law\bclr{ \sqrt{\mu^{-1} M} \bZ}, \law\bclr{ \sqrt{\mu^{-1} N} \bZ} } \label{eq:mnwass2} \\
	&\qquad +\done\bbclr{\law\bclr{ \sqrt{\mu^{-1} N} \bZ},  \law\bclr{ \sqrt{E} \bZ}}. \label{eq:expgeowass3}
}
We use below that if $X,Y,Z$ are random elements defined on the same space, then
\ben{\label{eq:condwass}
\done\bclr{\law(X),\law(Y)}\leq \IE\bcls{ \done\bclr{\law(X|Z),\law(Y|Z)}},
}
which easily follows from the definition of Wasserstein distance given at~\eq{eq:wassdef}.

To bound~\eq{eq:cltwass1}, we use the smooth function CLT, Corollary~\ref{cor:doneclt} below, 
which says that there is a constant $C_r$ such that for all $m\geq1$, 
\be{
\done\bbclr{\law\bbclr{\mu^{-1/2} \sum_{i=1}^{m} \bX_i}, \law\bclr{ \sqrt{\mu^{-1} m} \bZ}} \leq C_r \mu^{-1/2} m^{1/3} \IE\bcls{ \norm{\bX_1}_1^3}^{1/3}.
}
Conditioning on $M$, applying~\eq{eq:condwass}, and using independence, we find that
\besn{\label{eq:cltwass1a}
\done\bbclr{\law\bbclr{\mu^{-1/2} \sum_{i=1}^{M} \bX_i}, \law\bclr{ \sqrt{\mu^{-1} M} \bZ}} 	
	&\leq C_r \mu^{-1/2} \IE\bcls{ M^{1/3} }\IE\bcls{ \norm{\bX_1}_1^3}^{1/3} \\
	& \leq C_r \mu^{-1/2} \mu^{1/3} \IE\bcls{ \norm{\bX_1}_1^3}^{1/3},
}
where we have used Jensen's inequality in the last line.

To bound~\eq{eq:mnwass2}, we use~\eq{eq:condwass} and the general fact that for fixed $\bz\in\IR^d$, and random variables~$X,Y$,
\ben{\label{eq:mvarscalewass}
\done\bclr{\law(X\bz),\law(Y\bz)}\leq \norm{\bz}_1 \done\bclr{\law(X),\law(Y)}.
}
Now use the dual definition of Wasserstein distance, see, for example \cite{Rachev1998} or \cite{Villani2009}, to choose a coupling between $M$ and $N$ such that $\done(\law(M),\law(N))=\IE\bcls{\abs{M-N}}$. Using that $M,N\geq 1$, we have
\bes{
  \done\bclr{\law\bclr{\sqrt{M}},\law\bclr{\sqrt{N}}}
  & \leq \IE\babs{\sqrt{M}-\sqrt{N}} \leq \IE\bcls{\babs{\sqrt{M}-\sqrt{N}}\bclr{\sqrt{M}+\sqrt{N}}} \\ 
  & =\IE\bcls{\abs{M-N}} = \done(\law(M),\law(N)).
}
Thus, conditioning on $\bZ$, using~\eq{eq:condwass} and independence, we have that~\eq{eq:mvarscalewass} implies
\besn{\label{eq:mnwass2a}
\done\bbclr{ \law\bclr{ \sqrt{\mu^{-1} M} \bZ}, \law\bclr{ \sqrt{\mu^{-1} N} \bZ} }
	 &\leq \IE\bcls{\norm{\bZ}_1}\mu^{-1/2}  \done\bclr{\law\bclr{\sqrt{M}},\law\bclr{\sqrt{N}}} \\
	&\leq \mu^{-1/2} \bbbclr{\sum_{i=1}^r \sqrt{\Sigma_{ii}}} \done\bclr{\law(M),\law(N)},
}
where the last inequality is because $\IE\bcls{|Z_i|}\leq \sqrt{\Sigma_{ii}}$.

To bound~\eq{eq:expgeowass3}, we again use~\eq{eq:condwass},~\eq{eq:mvarscalewass}, and independence, to find 
\be{
\done\bbclr{\law\bclr{ \sqrt{\mu^{-1} N} \bZ},  \law\bclr{ \sqrt{E} \bZ}}\leq \IE\bcls{\norm{\bZ}_1} 
\done\bbclr{\law\bclr{ \sqrt{\mu^{-1} N} },  \law\bclr{ \sqrt{E} }}.
}
To bound this last Wasserstein distance, we use the usual coupling of a geometric to an exponential. If $N'=\ceil{E/(-\log(1-\mu^{-1}))}$, then $\law(N)=\law(N')$, and 
\besn{\label{eq:w33}
\done\bbclr{&\law\bclr{ \sqrt{\mu^{-1} N} },  \law\bclr{ \sqrt{E} }} \\
&\leq\IE\bbabs{ \sqrt{\mu^{-1} N'}-\sqrt{E} } \\
&=\IE\bbbcls{\bbabs{ \sqrt{\mu^{-1} N'}-\sqrt{E} } \Ind{E<\mu^{-1}}}+\IE\bbbcls{\bbabs{ \sqrt{\mu^{-1} N'}-\sqrt{E} } \Ind{E>\mu^{-1}}}.
}
Since $\sqrt{\mu^{-1} N'}\geq \mu{-1}$, if  $E>\mu^{-1}$, then  
\be{
\bbabs{\sqrt{\mu^{-1} N'}-\sqrt{E}} \leq \frac{\sqrt{\mu}}{2} \babs{\mu^{-1} N'-E},
}
and so
\besn{\label{eq:t22}
\IE\bbbcls{\bbabs{ \sqrt{\mu^{-1} N'}-\sqrt{E} } \Ind{E>\mu^{-1}}}
	&\leq \frac{\sqrt{\mu}}{2} \IE\babs{\mu^{-1} N'-E} \\
	&\leq \frac{\sqrt{\mu}}{2} \IE\bbabs{\frac{E}{\mu\log(1-\mu)}+\mu^{-1}-E} \\
	&\leq \frac{\sqrt{\mu}}{2} \left(\bbabs{1+\frac{1}{\mu\log(1-\mu)}}+\mu^{-1}\right)\\
	&\leq \mu^{-1/2},
}
since $\abs{1+1/(\mu\log(1-\mu))}\leq 1$ for $\mu>1$. Finally,
\besn{\label{eq:s22}
\kern-0.5em\IE\bbbcls{\bbabs{ \sqrt{\mu^{-1} N'}-\sqrt{E} } \Ind{E<\mu^{-1}}}
	&\leq\left(\sqrt{\mu^{-1} \left(-\frac{1}{\mu\log(1-\mu)}+1\right) }+ \mu^{-1/2}\right) \IP(E<\mu^{-1}) \\
	&\leq \frac{\sqrt{2}+1}{\mu^{3/2}}\leq \frac{2.5}{\mu^{1/2}},
}
since $1-1/(\mu\log(1-\mu))\leq 2$, $\IP(E<\mu^{-1})\leq \mu^{-1}$, and $\mu>1$.
Combining~\eq{eq:w33},~\eq{eq:t22}, and~\eq{eq:s22} with~\eq{eq:mnwass2a} and~\eq{eq:cltwass1a} yields the result.
\end{proof}

\subsection{Proof of Theorem~\ref{t3}}

As in the proof of Theorem~\ref{t1}, we use Lemma~\ref{lem:indwassbd}
to move the problem to a random sum of i.i.d.\ random variables (with cost to the error term). 
Then we apply Theorem~\ref{thm:mvrenyi} to the random sum. 
Our first lemma gives some moment information for the summands, in particular, it shows 
the covariances converge, and the fourth moments are appropriately bounded.

Recall the notation and constructions for the BRW and modified BRW, and write $\mu_{n}(j):=\IE[M_{n}(j)]$.
\begin{lemma}\label{lem:14} 
Assume the hypotheses of Theorem~\ref{t3}, and for $j, k\in \{1,\ldots, r\}$, denote
\be{
A_n(j,k):=\Cov\bclr{ M_{n}(j)-\mu_{n}(j)Z_n, M_{n}(k)-\mu_{n}(k)Z_n}.
}
Then the limits 
\be{
  \lim_{n\toinf} A_n(j,k) =: \Sigma_{jk}
}
exist, are finite, and
\be{
  \abs{A_n(j,k) - \Sigma_{jk}}\leq cn^{2-d/3}.
}
Furthermore,
\ben{\label{eq:4momlap}
\lim_{n\to\infty}\frac{1}{n}\IE\bcls{\clr{M_{n}(j)-\mu_{n}(j)Z_n}^4}=3 \sigma^2 \Sigma_{jj}^2.
}
\end{lemma}
\begin{remark}
As a check on the limiting constant and linear growth of the fourth moment of $\clr{M_{n}(j)-\mu_{n}(j)Z_n}$ given
by~\eq{eq:4momlap}, Theorem~\ref{t3} suggests (assuming appropriate uniform integrability) that for $E\sim\Exp(1)$ and $Z_j\sim\rN(0,\wt \Sigma_{jj})$,
\ben{\label{eq:598}
n^{-2}\IE\bcls{\clr{M_{n}(j)-\mu_{n}(j)Z_n}^4 | Z_n>0} \to \IE\cls{E^2} \IE\cls{Z_{j}^4}=6 \tilde \Sigma_{jj}^2.
}
The left hand side of~\eq{eq:598} is equal to
\ben{\label{eq:599}
n^{-1}\frac{\IE\bcls{\clr{M_{n}(j)-\mu_{n}(j)Z_n}^4}}{n \IP( Z_n>0) }\sim n^{-1}\frac{\IE\bcls{\clr{M_{n}(j)-\mu_{n}(j)Z_n}^4}}{2/\sigma^2 }.
}
From~\eq{eq:r4} below, we see that $\wt \Sigma= (\sigma^2/2)\Sigma$, which, with~\eq{eq:598} and~\eq{eq:599}, 
agrees with~\eq{eq:4momlap}.
\end{remark}
\begin{proof}[Proof of Lemma~\ref{lem:14}]
Write
  \bes{
A_n(j,k)
	&=\IE[M_n(j)M_n(k)] - \mu_n(k) \IE[M_n(j)Z_n]- \mu_n(j) \IE[ M_n(k)Z_n]+\mu_n(j)\mu_n(k) \IE[ Z_n^2].
	}
Let $(M_{n-1}^i(j),M_{n-1}^i(k),Z_{n-1}^i)_{i=1}^{Z_1}$ be defined as follows: The first (respectively, second) coordinate is the number of sites with $j$ (respectively, $k$) particles in generation~$n$ descended from particle~$i$ in generation~$1$, and the third coordinate is the number of offspring in generation $n$ descended from generation~$1$. Note that, given~$Z_1$, these are i.i.d.\ copies of $(M_{n-1}(j), M_{n-1}(k), Z_{n-1})$. 
Now write 
\ba{
\IE[M_n(j)M_n(k)] &=\IE\bbbcls{\sum_{i,\ell=1}^{Z_1} M_{n-1}^i(j) M_{n-1}^\ell(k)}
	+\IE\bbbcls{M_n(j)M_n(k)-\sum_{i,\ell=1}^{Z_1} M_{n-1}^i(j) M_{n-1}^\ell(k)}  \\
	&=:\IE\bbbcls{\sum_{i,\ell=1}^{Z_1} M_{n-1}^i(j) M_{n-1}^\ell(k)} + e\sp{1}_{n}(j,k). 
}
The first term above is the main contribution, and the second term is a small error. For the first term,
\bes{
  \IE\bbbcls{\sum_{i,\ell=1}^{Z_1} M_{n-1}^i(j) M_{n-1}^\ell(k)}
  &=\IE\bbbcls{\sum_{i,\ell=1}^{Z_1} \IE\bcls{M_{n-1}^i(j) M_{n-1}^\ell(k)| Z_1}} \\
  &=\IE\bbbcls{\sum_{i=1}^{Z_1} \IE\bcls{M_{n-1}^i(j) M_{n-1}^i(k)| Z_1} + \sum_{i\neq j}^{Z_1} \IE\bcls{M_{n-1}^i(j) M_{n-1}^j(k)| Z_1}} \\
	&= \IE[Z_1]\IE\bcls{M_{n-1}^1(j) M_{n-1}^1(k)} + \IE[Z_1(Z_1-1)] \mu_{n-1}(j)\mu_{n-1}(k)\\[1ex]
	&= \IE\bcls{M_{n-1}^1(j) M_{n-1}^1(k)} + \sigma^2 \mu_{n-1}(j)\mu_{n-1}(k). 
}
Similarly,
\bes{
\IE[M_n(j)Z_n]&=\IE\bbbcls{\sum_{i,\ell=1}^{Z_1} M_{n-1}^i(j) Z_{n-1}^\ell}
	+\IE\bbbcls{Z_n\bbclr{M_n(j)-\sum_{i=1}^{Z_1} M_{n-1}^i(j)} }, \\
	&=:\IE\bbbcls{\sum_{i,\ell=1}^{Z_1} M_{n-1}^i(j) Z_{n-1}^\ell}
	+e\sp{2}_{n}(j),
}
and 
\be{
\IE\bbbcls{\sum_{i,\ell=1}^{Z_1} M_{n-1}^i(j) Z_{n-1}^\ell}
	=\IE\bcls{M_{n-1}^1(j) Z_{n-1}^1} + \sigma^2 \mu_{n-1}(j).
}
Finally,
\be{
\IE[Z_n^2] = \IE[( Z_{n-1})^2] + \sigma^2.
}
Using that $\IE[M_n(\ell) Z_n]\leq \IE[Z_n^2] \leq c n$ and, from the proof of  \cite[Proposition 21]{Lalley2011}, that $\abs{\mu_{n}(\ell)-\mu_{n-1}(\ell)}\leq n^{-d/2}$, we can collect the work above to find
\be{
\abs{A_{n}(j,k)-A_{n-1}(j,k)}\leq c\bclr{\babs{e_n\sp{1}(j,k)}+\babs{e_n\sp{2}(j)}+\babs{e_n\sp{2}(k)} + n^{1-d/2}}.
}
To bound the errors, first note
\ban{
\babs{e_n\sp{1}(j,k)}&\leq \IE\bbbcls{M_{n}(j)\bbabs{M_n(k)-\sum_{\ell=1}^{Z_1} M_{n-1}^\ell(k)}}+ \IE\bbbcls{\sum_{\ell=1}^{Z_1} M_{n-1}^\ell(k)\bbabs{M_n(j)-\sum_{i=1}^{Z_1} M_{n-1}^i(j)}}\notag\\
	&\leq  \IE\bbbcls{Z_n\bbabs{M_n(k)-\sum_{\ell=1}^{Z_1} M_{n-1}^\ell(k)}}+ \IE\bbbcls{Z_n\bbabs{M_n(j)-\sum_{i=1}^{Z_1} M_{n-1}^i(j)}} \label{eq:111} \\
	&\leq 2 \IE[ Z_n Y_{n;1}]. \notag
}
Similarly,
\bes{
\abs{e_n\sp{2}(j)}&\leq \IE\bbbcls{Z_n\bbabs{M_n(j)-\sum_{i=1}^{Z_1} M_{n-1}^i(j)}} \\
	&\leq  \IE[ Z_n Y_{n;1}],
}
and the same inequality holds with $\abs{e_n\sp{2}(k)}$ on the left hand side. 
For $\alpha>0$, to be chosen later, we bound
\besn{\label{eq:221}
\IE[Z_n Y_{n;1}]&\leq \IE\bcls{Z_n Y_{n;1} \Ind{Z_n > n^{1+\alpha}}} + n^{1+\alpha} \IE[Y_{n;1}]  \\
	&\leq \IE\bcls{Z_n^2 \Ind{Z_n > n^{1+\alpha}}} + c n^{1+\alpha-d/2},
}
where we have used \cite[Corollary 20]{Lalley2011}.
Now squaring both sides of the inequality in the indicator and bounding by the ratio, we have
\bes{
\IE\bcls{Z_n^2 \Ind{Z_n > n^{1+\alpha}}}
	&\leq \frac{\IE[Z_n^4]}{n^{2(1+\alpha)}} \\	
	&\leq c n^{1-2\alpha},
}
since $\IE[Z_n^4]= \bigo(n^3)$. Thus, choosing $\alpha=d/6$, we obtain $\IE[Z_n Y_{n;1}]\leq c n^{1-d/3}$, and thus
\be{
\abs{A_{n}(j,k)-A_{n-1}(j,k)}\leq cn^{1-d/3}.
}
Since $d\geq 7$, $\sum_{n\geq 1} n^{1-d/3}<\infty$, and therefore $(A_{n}(j,k))_{n\geq1}$ is a Cauchy sequence; denote its limit by $\Sigma_{jk}$, and observe that
\be{
\abs{A_{n}(j,k)-\Sigma_{jk}}\leq cn^{2-d/3}.
} 

To prove the second assertion, we fix~$j$ and drop it from the notation, e.g., writing $M_n$ for $M_n(j)$.
We follow the strategy above, but now there are higher moments, which we denote by
\ba{
\mu_{i}\spm{k,\ell}:= \IE\bcls{(M_{i})^{k} (Z_{i})^{\ell}},
}
for $i\in\{n-1,n\}$.
Now, expanding and then swapping the coefficients with powers of $\mu_{n}\spm{0,1}$ for those with $\mu_{n-1}\spm{0,1}$, we have
\besn{\label{eq:4mom1}
\IE\bcls{\clr{M_{n}-\mu_{n}\spm{1,0} Z_n}^4}&= \mu_{n}\spm{4,0}- 4\mu_{n}\spm{1,0}\mu_{n}\spm{3,1}+ 6\bclr{ \mu_{n}\spm{1,0}}^2\mu_{n}\spm{2,2} \\
	& \qquad\qquad-4\bclr{ \mu_{n}\spm{1,0}}^3\mu_{n}\spm{1,3}+ \bclr{ \mu_{n}\spm{1,0}}^4 \mu_{n}\spm{0,4} \\
	&= \mu_{n}\spm{4,0}- 4\mu_{n-1}\spm{1,0}\mu_{n}\spm{3,1}+ 6\bclr{ \mu_{n-1}\spm{1,0}}^2\mu_{n}\spm{2,2} \\
	& \qquad\qquad-4\bclr{ \mu_{n-1}\spm{1,0}}^3\mu_{n}\spm{1,3}+ \bclr{ \mu_{n-1}\spm{1,0}}^4 \mu_{n}\spm{0,4} + \bigo(n^{3-d/2}),
}
where the second equality follows, similar to the argument above, from
\ben{\label{eq:221122}
\babs{\bclr{\mu_{n}\spm{1,0}}^k \mu_{n}\spm{4-k,k}-\bclr{\mu_{n-1}\spm{1,0}}^k \mu_{n}\spm{4-k,k}}
	\leq c \IE[Z_n^4] \abs{\mu_{n}\spm{1,0}- \mu_{n-1}\spm{1,0}}  \leq c n^{3-d/2},
}
where $k=0,1,\ldots,4$. Now, denote the error made in the summands above when replacing the moments by the first step moments as
\be{
e_n\sp{k,4-k}:=a_{k} \bclr{\mu_{n-1}\sp{1,0}}^{k}\bbbclr{ \mu_{n}\spm{k,4-k}-\IE\bbbcls{\bbbclr{\sum_{i=1}^{Z_1} M_{n-1}^i}^k\bbclr{\sum_{j=1}^{Z_1} Z_{n-1}^j}^{4-k}}}, 
}
 where $a_k$ is $1,-4$, or $6$ as appropriate. 
 Bounding these similar to~\eq{eq:111} and~\eq{eq:221} above, using that $\IE[X^{5+\floor{18/(d-6)}}]<\infty$, we have that for $\alpha>0$ and $\beta= \floor{18/(d-6)}+1$,
 \bes{
 \abs{e_n\sp{k,4-k}} &\leq c \IE[ Z_n^{3} Y_{n;1}] \\
	 	&\leq c \bclr{\IE[Z_n^{4+\beta}] n^{-\beta(1+\alpha)}+ n^{3(1+\alpha)-d/2})}\\
 	&\leq c  (n^{3-\beta \alpha}+ n^{3(1+\alpha)-d/2})
 }
 Choosing $\alpha = (d-6)/6-\eps$, for $\eps>0$ small enough that $\alpha \beta>3$, and noting that $d\geq 7$, we have
 \be{
 \abs{e_n\sp{k,4-k}}=\bigo(n^{-\delta}), 
 }
 for $\delta=\min\{\alpha\beta-3,3\eps\}>0$.
 Thus we find the fourth moment~\eq{eq:4mom1} is equal to 
\besn{\label{eq:4momexp}
	&\IE\bbcls{\bbclr{\sum_{i=1}^{Z_1} M_{n-1}^i}^4}- 4\mu_{n-1}\spm{1,0}\IE\bbbcls{\bbclr{\sum_{i=1}^{Z_1} M_{n-1}^i}^3\bbclr{\sum_{j=1}^{Z_1} Z_{n-1}^j}} \\
	&\quad+ 6\bclr{ \mu_{n-1}\spm{1,0}}^2 \IE\bbbcls{\bbclr{\sum_{i=1}^{Z_1} M_{n-1}^i}^2\bbclr{\sum_{j=1}^{Z_1} Z_{n-1}^j}^2} 	-4\bclr{ \mu_{n-1}\spm{1,0}}^3 \IE\bbbcls{\bbclr{\sum_{i=1}^{Z_1} M_{n-1}^i}\bbclr{\sum_{j=1}^{Z_1} Z_{n-1}^j}^3}\\
	& \qquad+ \bclr{ \mu_{n-1}\spm{1,0}}^4 \IE\bbbcls{\bbclr{\sum_{j=1}^{Z_1} Z_{n-1}^j}^4} 
		+\bigo(n^{-\delta'}).
}
for $\delta'=\min\{\delta,3-d/2\}>0$. 
As before, we expand the random sums in the expectations above and then simplify. We cover in detail 
only the middle term, which is the most involved, and just write the final expressions for the other terms.
 Write $\Sigma_{\{i,j,k\}}$ for sums over distinct indices.
For the middle term,
\ba{
\bbclr{\sum_{i=1}^{Z_1} &M_{n-1}^i}^2\bbclr{\sum_{j=1}^{Z_1} Z_{n-1}^j}^2 \\
	& = \sum_{ i_1,i_2,i_3,i_4=1 }^{Z_1} M_{n-1}^{i_1} M_{n-1}^{i_2} Z_{n-1}^{i_3} Z_{n-1}^{i_4}  \\
	&= \sum_{i=1}^{Z_1} \bclr{ M_{n-1}^{i}}^2 \bclr{ Z_{n-1}^{i}}^2  
		+  \sum_{\{i_1,i_2\}=1}^{Z_1} \bclr{ M_{n-1}^{i_1}}^2 \bclr{ Z_{n-1}^{i_2}}^2  \\
	&\qquad+\sum_{ \{i_1,i_2,i_3\}=1 }^{Z_1} \bclr{M_{n-1}^{i_1}}^2 Z_{n-1}^{i_2} Z_{n-1}^{i_3} 
		+ 2 \sum_{\{i_1,i_2\}=1}^{Z_1} \bclr{M_{n-1}^{i_1}}^2 Z_{n-1}^{i_1} Z_{n-1}^{i_2}  \\
	&\qquad +  \sum_{\{ i_1,i_2,i_3,i_4\}=1 }^{Z_1} M_{n-1}^{i_1} M_{n-1}^{i_2} Z_{n-1}^{i_3} Z_{n-1}^{i_4} 
		+ 4 \sum_{ \{i_1,i_2,i_3\}=1 }^{Z_1}M_{n-1}^{i_1} M_{n-1}^{i_2} Z_{n-1}^{i_1} Z_{n-1}^{i_3}  \\
	&\qquad +  2 \sum_{\{ i_1,i_2\}=1 }^{Z_1}M_{n-1}^{i_1} M_{n-1}^{i_2} Z_{n-1}^{i_1} Z_{n-1}^{i_2} 
		+ 2 \sum_{ \{i_1,i_2\}=1 }^{Z_1}M_{n-1}^{i_1} M_{n-1}^{i_2} \bclr{Z_{n-1}^{i_1}}^2  \\
	&\qquad +  \sum_{\{ i_1,i_2,i_3\}=1 }^{Z_1}M_{n-1}^{i_1} M_{n-1}^{i_2} \bclr{Z_{n-1}^{i_3}}^2.
}
For a quick parity check of this formula, note that for non-negative integer $z$, 
\be{
z^4=z(z-1)(z-2)(z-3)+6 z(z-1)(z-2)+7z(z-1) + z.
}
Now, taking expectation by first conditioning on $Z_1$, writing $\gamma_k:=\IE[Z_1(Z_1-1)\cdots(Z_1-k+1)]$, $k=3,4$, we have
\ba{
 \IE\bbbcls{\bbclr{\sum_{i=1}^{Z_1} M_{n-1}^i}^2\bbclr{\sum_{j=1}^{Z_1} Z_{n-1}^j}^2}&=\mu_{n-1}\spm{2,2}+\sigma^2 \mu_{n-1}\spm{2,0}\mu_{n-1}\spm{0,2}+ \gamma_3 \mu_{n-1}\spm{2,0} + 2 \sigma^2 \mu_{n-1}\spm{2,1} \\
			&\qquad + \gamma_4 \bclr{ \mu_{n-1}\spm{1,0}}^2 + 4 \gamma_3\mu_{n-1}\spm{1,1}\mu_{n-1}\spm{1,0} + 2 \sigma^2 \bclr{\mu_{n-1}\spm{1,1}}^2 + 2 \sigma^2\mu_{n-1}\spm{1,2}\mu_{n-1}\spm{1,0} \\[2ex]
			&\qquad + \gamma_3 \bclr{\mu_{n-1}\spm{1,0}}^2 \mu_{n-1}\spm{0,2}.
	}
Similar arguments shows
\ba{
\IE\bbcls{\bbclr{\sum_{i=1}^{Z_1} M_{n-1}^i}^4}
	&= \mu_{n-1}\spm{4,0} + 3 \sigma^2 \bclr{\mu_{n-1}\spm{2,0}}^2 
				+ 6 \gamma_3 \mu_{n-1}\spm{2,0}\bclr{\mu_{n-1}\spm{1,0}}^2  \\
				&\qquad	+4 \sigma^2 \mu_{n-1}\spm{3,0} \mu_{n-1}\spm{1,0} + \gamma_4 \bclr{\mu_{n-1}\spm{1,0}}^4, \\
\IE\bbbcls{\bbclr{\sum_{i=1}^{Z_1} M_{n-1}^i}^3\bbclr{\sum_{j=1}^{Z_1} Z_{n-1}^j}}
	&=\mu_{n-1}\spm{3,1}+ 3 \gamma_3\mu_{n-1}\spm{1,1}\bclr{\mu_{n-1}\spm{1,0}}^2 
					+3 \sigma^2 \mu_{n-1}\spm{2,1}\mu_{n-1}\spm{1,0}\\
		&\qquad+3 \sigma^2 \mu_{n-1}\spm{1,1}\mu_{n-1}\spm{2,0}
				+\sigma^2 \mu_{n-1}\spm{3,0} + 3 \gamma_3 \mu_{n-1}\spm{2,0}\mu_{n-1}\spm{1,0} 
					+   \gamma_4\bclr{\mu_{n-1}\spm{1,0}}^3,\\
\IE\bbbcls{\bbclr{\sum_{i=1}^{Z_1} M_{n-1}^i}\bbclr{\sum_{j=1}^{Z_1} Z_{n-1}^j}^3}
	&=\mu_{n-1}\spm{1,3}+ 3 \gamma_3\mu_{n-1}\spm{1,1} 
			+3 \sigma^2 \mu_{n-1}\spm{1,2}+3 \sigma^2 \mu_{n-1}\spm{1,1}\mu_{n-1}\spm{0,2}
				\\
		&\qquad+\sigma^2 \mu_{n-1}\spm{0,3}\mu_{n-1}\spm{1,0} + 3 \gamma_3 \mu_{n-1}\spm{0,2}\mu_{n-1}\spm{1,0} 
				+   \gamma_4\mu_{n-1}\spm{1,0},\\
\IE\bbbcls{\bbclr{\sum_{j=1}^{Z_1} Z_{n-1}^j}^4}
	&=\mu_{n-1}\spm{0,4}+3 \sigma^2 \bclr{\mu_{n-1}\spm{0,2}}^2 + 6 \gamma_3 \mu_{n-1}\spm{0,2}+ 4 \sigma^2 \mu_{n-1}\spm{0,3} + \gamma_4.
}
Plugging these into~\eq{eq:4momexp}, we find  (it's easiest to compute the coefficients for each of  $\sigma^2, \gamma_3,\gamma_4$; the last two are zero) that
\be{
\IE\bcls{\clr{M_{n}-\mu_{n}Z_n}^4}-\IE\bcls{\clr{M_{n-1}-\mu_{n-1}Z_{n-1}}^4}= 3 \sigma^2 A_{n-1}(j,j)^2 + \bigo \bclr{n^{-\delta'}}.
}
Therefore,
\be{
\IE\bcls{\clr{M_{n}-\mu_{n}Z_n}^4}= 3 n \sigma^2 A_{n-1}(j,j)^2 + \bigo\bclr{1+n^{1-\delta'}},
}
and the result easily follows.
\end{proof}

\begin{proof}[Proof of Theorem~\ref{t3}]
Assume $d\geq 7$ and fix $n\geq 3$ and 
\[
1\leq m:= \bfloor{n-n^{\frac{20}{3(2d+1)}}} <n.
\]
Recall the definition and constructions for the BRW and modified BRW, and denote $\mu_{n,m}(j):=\IE[M_{n,m}^i(j)]=\IE[M_{n-m}(j)]$, $\wt M_{n,m}^i(j):=M_{n,m}^i(j)-\mu_{n,m}(j)Z_{n,m}^i$. Now, let $\wt\Sigma= (\sigma^2/2) \Sigma$ and use the triangle inequality to find
\ban{
\done&\bbbclr{\law\bbbclr{\bbbclr{\frac{M_n(j)-\kappa_j Z_n}{\sqrt{n}}}_{j=1}^r \Big|Z_n>0},\SL_r(\wt\Sigma)} \notag\\
	&\leq
\done\left(\law\bbbclr{\bbbclr{\frac{\hat M_n(j)-\kappa_j \hat Z_n}{\sqrt{n}}}_{j=1}^r },\law\bbbclr{\bbbclr{\frac{\sum_{i=1}^{\hat Z_m}M_{n,m}^i(j)-\kappa_j \hat Z_n}{\sqrt{n}}}_{j=1}^r}\right) \label{eq:r2}\\
\begin{split}\label{eq:r3}
	&\quad+\done\bbbbclr{\law\bbbclr{\bbbclr{\frac{\sum_{i=1}^{\hat Z_m}M_{n,m}^i(j)-\kappa_j \hat Z_n}{\sqrt{n}}}_{j=1}^r}, \\
	&\hspace{3cm}\law\bbbclr{\bbbclr{\frac{\sum_{i=1}^{\hat Z_m}M_{n,m}^i(j)-\mu_{n,m}(j) \hat Z_n}{\sqrt{n}}}_{j=1}^r}} 
\end{split}   \\
\begin{split}\label{eq:r3a}
		&\quad+\done\bbbbclr{\law\bbbclr{\bbbclr{\frac{\sum_{i=1}^{\hat Z_m}M_{n,m}^i(j)-\mu_{n,m}(j) \hat Z_n}{\sqrt{n}}}_{j=1}^r}, \\
		&\hspace{3cm}
		\law\bbbclr{\bbbclr{\frac{\sum_{i=1}^{\hat Z_m}\left(M_{n,m}^i(j)-\mu_{n,m}(j)Z_{n,m}^i\right)}{\sqrt{n}}}_{j=1}^r}}
\end{split}		\\
	&\quad+\done\left(\law\bbbclr{\bbbclr{\frac{\sum_{i=1}^{\hat Z_m}\wt M_{n,m}^i(j)}{\sqrt{n}}}_{j=1}^r},\SL_r\bclr{(\sigma^2/2)\Sigma_n}\right) \label{eq:r4} \\	
  &\quad +\frac{\sigma}{\sqrt{2}} \, \done\bclr{\SL_r\clr{\Sigma_n},\SL_r\clr{\Sigma}}. \label{eq:r5}
}
where $\Sigma_{n} = (A_{n-m}(j,k))_{j,k}$ with $A_{n-m}(j,k) = \Cov(\wt M_{n,m}^1(j),\wt M_{n,m}^1(k))$. 
Using the coupling definition of Wasserstein distance and Lemma~\ref{lem:indwassbd}, we can bound~\eq{eq:r2} by noting
\be{
n^{-1/2}  \IE\left|\hat M_n(j)-\sum_{i=1}^{\hat Z_m}M_{n,m}^i(j)\right|\leq c n^{-\frac{2d-9}{6(2d+1)}}.
}
Similarly,~\eq{eq:r3} is bounded from
\be{
n^{-1/2}\babs{\mu_{n,m}(j) -\kappa_j} \IE[\hat Z_n]\leq c n^{-\frac{14d-3}{6(2d+1)}}\leq c n^{-\frac{2d-9}{6(2d+1)}},
}
where in the first inequality we have used Lemma~\ref{l1} and that $\IE[\hat Z_n]=\bigo(n)$.
For~\eq{eq:r3a}, note that
\bes{
n^{-1/2}\IE&\bbbabs{\sum_{i=1}^{\hat Z_m}M_{n,m}^i(j)-\mu_{n,m}(j) \hat Z_n -\sum_{i=1}^{\hat Z_m}\left(M_{n,m}^i(j)-\mu_{n,m}(j)Z_{n,m}^i\right)} \\
&\hspace{1cm}\leq n^{-1/2}\mu_{n,m}(j)\IE\bbabs{\hat Z_n -\sum_{i=1}^{\hat Z_m}Z_{n,m}^i} \\
&\hspace{1cm}\leq cn^{-1/2}\IE\bbcls{ \sum_{i=m+1}^n R_{n,i}' +1} \leq c(n-m) 
	\leq c  n^{-\frac{6d-37}{6(2d+1)}}\leq c  n^{-\frac{2d-9}{6(2d+1)}},
}
where we used Lemma~\ref{l1} in the second inequality, and part $(iii)$ of Lemma~\ref{lem:keyests} in the second to last.
Noting our Wasserstein distance is with respect to $L_1$ distance, summing over $j$ and $k$, and using the inequalities above shows that~\eq{eq:r2},~\eq{eq:r3}, and~\eq{eq:r3a} are upper bounded by $c  n^{-\frac{2d-9}{6(2d+1)}}$. 

To bound~\eq{eq:r4}, we apply Theorem~\ref{thm:mvrenyi} with $M=\hat Z_m$, $\bX_i=(\wt M_{n,m}^i(j))_{j=1}^r$, and $\mu=\IE[\hat Z_m]$, which, using~\eq{eq:diffcond}, is of strict order $n$. Lemma~\ref{lem:14} states that
\be{
\IE\bbcls{\bclr{\wt M_{n,m}^1(j)}^4}\leq c (n-m),
}
and then repeated use of H\"older's and Jensen's inequalities implies that
\be{
 \IE\bbcls{ \bnorm{\bclr{\wt M_{n,m}^1(j)}_{j=1}^r}_1^3}^{1/3}\leq c (n-m)^{1/4}\leq c n^{-\frac{5}{3(2d+1)}}.
}
Multiplying this by the $n^{-1/6}$ factor coming from the powers of $\mu$ in the bound from Theorem~\ref{thm:mvrenyi} gives
a term of order $n^{-\frac{2d-9}{6(2d+1)}}$. For the remaining (nontrivial) term, the triangle inequality implies 
\ba{
\done\bclr{\law(\hat Z_m),\Geo(\mu^{-1})}&\leq
	\done\bclr{\law(\hat Z_m),\law(Z_m|Z_m>0)}+\done\bclr{\law(Z_m|Z_m>0),\Exp\bclr{\sigma^2/(2m)}}  \\
	&\qquad +\done\bclr{\Exp\bclr{\sigma^2/(2 m)},\Exp(\mu^{-1})}+\done\bclr{\Exp(\mu^{-1}),\Geo(\mu^{-1})}, \\
	&\leq c\bclr{ (n-m) + \log(m) + \babs{m(\sigma^2/2)-\mu}},
}
where the second inequality uses Lemma~\ref{l2}, Theorem~\ref{t2},  and standard couplings. Using Lemmas~\ref{lem:kolesterr} and~\ref{l2}, we have
\ba{
 \babs{m(\sigma^2/2)-\mu}&\leq  \babs{m(\sigma^2/2)-\IE[Z_m|Z_m>0]}+ \babs{\IE[Z_m|Z_m>0]-\IE[Z_m|Z_n>0]} \\
 	&\leq  c\bclr{\log(m)^2+ (n-m)},
 }
Noting that $\mu^{-1/2}(n-m)\leq c n^{-\frac{1}{2}+\frac{20}{3(2d+1)}}\leq cn^{-\frac{2d-9}{6(2d+1)}}$, and putting these bounds into Theorem~\ref{thm:mvrenyi} implies
that~\eq{eq:r4} is bounded by  $cn^{-\frac{2d-9}{6(2d+1)}}$.

Finally, we bound \eq{eq:r5}. Using the representation $\law(\sqrt{E}\bZ)=\SL_r(\Cov(\bZ))$ for $E$ distributed as an exponential with rate one, independent of $\bZ$, an $r$-dimensional multivariate normal, we apply Lemma~\ref{lem:compmvn} below and Lemma~\ref{lem:14}, and noting that $d\geq 7$, we conclude that
\be{
  \done\bclr{\SL_r\clr{\Sigma_n},\SL_r\clr{\Sigma}}
  \leq c\bbbclr{\sum_{j,k}\abs{A_{n-m}(j,k)-\Sigma_{jk}}}^{1/2}
  \leq cn^{-\frac{10(d-6)}{9(2d+1)}}\leq c  n^{-\frac{2d-9}{6(2d+1)}},
}
and combining the bounds above yields the theorem.
\end{proof}


\section{CLT with error}\label{sec:CLT}
In this section we prove a multivariate CLT with Wasserstein error for sums of i.i.d.\ variables that is adapted to
our setting. The proof is relatively standard using Stein's method, with the complications that we are working in the Wasserstein (rather than smoother test function) metric, and that we do not demand the covariance matrix be non-singular.

In what follows, denote by $\abs{\cdot}$ the Euclidean $L_2$-norm. For a $k$-times differential function $f:\IR^r\to\IR$, let
\be{
    M_k(f) := \sup_{x\in\IR^r}\sup_{\substack{a_1,\dots,a_k\in\IR^r:\\\abs{a_1}=\cdots=\abs{a_k}=1}}
        \bbbabs{\sum_{i_1,\dots,i_k=1}^r a_{1,i_1}\cdots a_{k,i_k} \frac{\partial^k f(x)}{\partial x_{i_1}\cdots\partial x_{i_r}}}.
}
Clearly, for any vectors $a_1,\dots,a_k,x\in\IR^r$, we have
\ben{\label{456}
    \bbbabs{\sum_{i_1,\dots,i_r=1}^r a_{1,i_1}\cdots a_{k,i_k} \frac{\partial^k f(x)}{\partial x_{i_1}\cdots\partial x_{i_r}}} \leq \abs{a_1}\cdots\abs{a_k}M_k(f).
}

\begin{theorem}\label{thm:doneclt}
Let $X_1,\dots,X_n$ be i.i.d.\ random vectors in $\IR^r$, with $\IE[X_1]=0$ and $\Var X_1 = \Sigma = (\Sigma_{uv})_{1\leq u,v\leq r}$. Let $W=n^{-1/2}\sum_{i=1}^n X_i$, and let $Z$ have a standard multivariate normal distribution, and let $Z_\Sigma = \Sigma^{1/2}Z$. Then, for any differentiable function $h:\IR^r\to\IR$,
\be{
    \abs{\IE[h(W)] - \IE[h(Z_\Sigma)]} \leq 
    2M_1(h)\bbbclr{\frac{2r\IE\bcls{\abs{X_1}^3}}{\sqrt{n}}}^{1/3}.
}
\end{theorem}
\begin{proof} We first replace $h$ by $h_\eps$, which is defined as
\be{
    h_{\eps}(x) = \IE[h(x+\eps Z')], \qquad x\in\IR^r,
}
where $Z'$ has a standard multivariate normal  distribution, independent of all else.
The error introduced by replacing $h$ by $h_\eps$ is at most
\be{
    \abs{\IE[h(W)]-\IE[h(W+\eps Z')]} \leq \eps M_1(h)\IE\bcls{\abs{Z'}} \leq \eps M_1(h)\sqrt{r},
}
and the same bound holds when $W$ is replaced by $Z_\Sigma$. Following \cite[Lemma~1]{Meckes2009}, since $h_\eps$ is infinitely differentiable, there is a function $f$ on $\IR^n$ such that
\be{
    \sum_{u,v=1}^r \Sigma_{uv}f_{uv}(x) - \sum_{u=1}^r x_uf_u(x) = h_\eps(x)-\IE[h_{\eps}(Z_\Sigma)],
    \qquad x\in\IR^r,
}
and from \cite[Lemma~2]{Meckes2009} and then \cite[Lemma~4.6]{Raic2018}, 
  this function satisfies
\be{
    M_3(f)\leq M_3(h_\eps)/3 \leq \eps^{-2}M_1(h)/3.
}
Now, letting $Y_i = n^{-1/2}X_i$, and using Taylor's expansion about $W^i = W-Y_i$, we have that
\ba{
    \IE&\sum_{u=1}^r W_uf_u(W)  \\
    & = \IE\sum_{i=1}^n \sum_{u=1}^r Y_{i,u} f_u(W) \\
    & = \IE\sum_{i=1}^n \sum_{u=1}^r Y_{i,u} \bbbcls{f_u(W^i)+\sum_{v=1}^r Y_{i,v} f_{uv}(W^i)
    + \sum_{v,w=1}^r Y_{i,v}Y_{i,w} \int_{0}^1 (1-s)f_{uvw}(W^i+sY_i)ds} \\
    & = \IE\sum_{i=1}^n \sum_{u,v=1}^r n^{-1}\Sigma_{uv} f_{uv}(W^i)
    + \IE\sum_{i=1}^n \sum_{u,v,w=1}^r Y_{i,u} Y_{i,v}Y_{i,w} \int_{0}^1 (1-s)f_{uvw}(W^i+sY_i)ds\\
    & = \IE\sum_{u,v=1}^r \Sigma_{uv} f_{uv}(W^1)
    + n\IE\sum_{u,v,w=1}^r Y_{1,u} Y_{1,v}Y_{1,w} \int_{0}^1 (1-s)f_{uvw}(W^1+sY_1)ds \\
    & = \IE\sum_{u,v=1}^r \Sigma_{uv} f_{uv}(W^1)
    + n\int_{0}^1 (1-s)\bbbclr{\IE\sum_{u,v,w=1}^r Y_{1,u} Y_{1,v}Y_{1,w} f_{uvw}(W^1+sY_1)}ds.
}
Applying \eq{456} to the quantity inside the second expectation, we have
\be{
    \bbbabs{\sum_{u,v,w=1}^r Y_{1,u} Y_{1,v}Y_{1,w} f_{uvw}(W^1+sY_1)} \leq \abs{Y_1}^3 M_3(f).
}
Now, let $Y_1'$ be an independent copy of $Y_1$, and note that we have $\Sigma_{uv} = n\IE(Y_{1,u}'Y_{1,v}')$. Hence, 
\bes{
    \IE\sum_{u,v=1}^r \Sigma_{uv}f_{uv}(W) 
    & = \IE\sum_{u,v=1}^r \Sigma_{uv} \bbbcls{f_{uv}(W^1)+\sum_{w=1}^r Y_{1,w} \int_0^1 f_{uvw}(W^1+sY_1)ds} \\
    & = \IE\sum_{u,v=1}^r \Sigma_{uv} f_{uv}(W^1)+\IE\sum_{u,v,w=1}^r \Sigma_{uv}Y_{1,w} \int_0^1 f_{uvw}(W^1+sY_1)ds \\
    & = \IE\sum_{u,v=1}^r \Sigma_{uv} f_{uv}(W^1)+n\int_0^1\bbbclr{\IE\sum_{u,v,w=1}^r Y_{1,u}'Y_{1,v}'Y_{1,w}  f_{uvw}(W^1+sY_1)}ds.
}
Applying again \eq{456} to the quantity inside the second expectation, we have
\be{
    \bbbabs{\sum_{u,v,w=1}^r Y'_{1,u} Y'_{1,v}Y_{1,w} f_{uvw}(W^1+sY_1)} \leq \abs{Y_1'}^2\abs{Y_1} M_3(f).
}
Subtracting one from the other, it follows that
\be{
    \abs{\IE[h_\eps(W)] - \IE[ h_\eps(Z_\Sigma)]}
     \leq \frac{\IE\bcls{\abs{X_1}^3}}{2\sqrt{n}}M_3(f) + \frac{\IE\bcls{\abs{X_1}^2}\IE\bcls{\abs{X_1}}}{\sqrt{n}}M_3(f)
    \leq \frac{M_1(h)\IE\bcls{\abs{X_1}^3}}{2\eps^{2}\sqrt{n}}.
}
Thus,
\be{
    \abs{\IE[h(W)] - \IE[h(Z_\Sigma)]} \leq 2\eps M_1(h)\sqrt{r} +  \frac{M_1(h)\IE\bcls{\abs{X_1}^3}}{2\eps^{2}\sqrt{n}},
}
and choosing 
\be{
    \eps = \bbbclr{\frac{\IE\bcls{\abs{X_1}^3}}{4\sqrt{rn}}}^{1/3}
}       
yields the final bound. 
\end{proof}

We also have the following easy corollary to fit our setup above. 
\begin{corollary}\label{cor:doneclt}
Under the notation and  assumptions of Theorem~\ref{thm:doneclt}, if $h$ is differentiable and $1$-Lipschitz with respect to the $1$-norm, then there is a constant~$C_r$ depending only on $r$, such that
\be{
  \abs{\IE[ h(W)] - \IE[h(Z_\Sigma)]} \leq  C_r \bbbclr{\frac{\IE\bcls{\norm{X_1}_1^3}}{\sqrt{n}}}^{1/3}.
}
\end{corollary}
\begin{proof}
By equivalence of norms in~$\IR^r$, there is a constant
$q_r$ depending only on $r$ such that for any $a\in\IR^r$, $q_r^{-1} \abs{a}\leq \norm{a}_1\leq q_r \abs{a}$. Therefore, for any $a,x\in\IR^r$, 
\be{
\bbbabs{\sum_{i=1}^r \frac{a_i}{\abs{a}} \frac{\partial h(x)}{\partial x_i}}\leq \frac{\norm{a}_1}{\abs{a}} \leq q_r,
}
and $\IE\bcls{\abs{X}^3}\leq q_r^3 \IE\bcls{\norm{X}_1^3}$. The result now follows from Theorem~\ref{thm:doneclt}.
\end{proof}

Finally, we state a simple lemma used to compare centered multivariate normal distributions.

\begin{lemma}\label{lem:compmvn} Let $\Sigma$ and $\Sigma'$ be two non-negative semi-definite $(r\times r)$ matrices for $r\geq 1$. Let $X=(X_1,\dots,X_r)$, respectively $\bY=(Y_1,\dots,Y_r)$, be a centered multivariate random normal vector with covariance matrix $\Sigma$, respectively $\Sigma'$. Then
\be{
  \dw\bclr{\law(\bX),\law(\bY)}\leq C\bbbclr{\sum_{u,v=1}^r\babs{\Sigma_{uv}-\Sigma'_{uv}}}^{1/2}
}
for some constant $C$ that only depends on $r$.
\end{lemma}
\begin{proof} 
Using Stein's identity for the multivariate normal, for any twice-differentiable function~$f$, we have
\bes{
  \IE\bbbclc{\sum_{u,v=1}^r \Sigma_{uv}f_{uv}(Y)-\sum_{u=1}^r X_uf_u(Y)} 
  	&=\IE\bbbclc{\sum_{u,v=1}^r \bclr{\Sigma_{uv}-\Sigma'_{uv}}f_{uv}(Y)}  \\
	&\leq M_2(f) \sum_{u,v=1}^r \babs{\Sigma_{uv}-\Sigma'_{uv}}.
}
The result now follows by using the smoothing argument and Stein's method as in the proof of Theorem~\ref{thm:doneclt}; we omit the details. 
\end{proof}


\section*{Acknowledgments}

AR was supported by Singapore Ministry of Education Academic Research Fund Tier 2 grant MOE2018-T2-2-076, and AR and NR were supported by Australian Research Council research grant DP150101459. We thank two referees for their helpful comments.


\end{document}